\documentclass[journal,twoside,web]{ieeecolor}
\usepackage{generic}

\usepackage{amsmath,amssymb,amsthm,mathtools}
\usepackage{xcolor}
\usepackage{graphics}
\usepackage[compress,noadjust]{cite}
\usepackage{epstopdf}
\usepackage{float}
\usepackage{soul}
\usepackage[integrals]{wasysym}
\usepackage{algorithm,algcompatible}
\usepackage{siunitx}
\usepackage[implicit=false]{hyperref}

\makeatletter
\newtheoremstyle{stylename}% name of the style to be used
  {\z@}% measure of space to leave above the theorem. E.g.: 3pt
  {\z@}% measure of space to leave below the theorem. E.g.: 3pt
  {}% name of font to use in the body of the theorem
  {\parindent}% measure of space to indent
  {\itshape}% name of head font
  {:}% punctuation between head and body
  {5\p@ plus\p@ minus\p@\relax}% space after theorem head; " " = normal interword space
  {\thmname{#1}\thmnumber{ #2}\thmnote{ (#3)}}
\def\QED{\mbox{\rule[0pt]{1.3ex}{1.3ex}}}

\newlength{\myparindent}
\makeatother
\theoremstyle{stylename}
\newtheorem{theorem}{Theorem}
\newtheorem{corollary}{Corollary}
\newtheorem{lemma}{Lemma}
\newtheorem{proposition}{Proposition}

\newtheorem{definition}{Definition}
\newtheorem{assumption}{Assumption}

\newtheorem{remark}{Remark}

\newcommand{\new}{\color{blue}}
\newcommand{\SymCap}{\mathrm{Cap}}

\begin{document}
\title{Data-driven control of switched linear systems with probabilistic stability guarantees}

\author{Zheming Wang, Guillaume O. Berger and Rapha\"el M. Jungers
\thanks{Z. Wang and R. Jungers are with the ICTEAM Institute, UCLouvain, Louvain-la-Neuve,1348, Belgium. R. Jungers is a  FNRS honorary Research Associate. Email: \{zheming.wang,raphael.jungers\}@uclouvain.be}
\thanks{G. Berger is with the CUPLV lab, University of Colorado Boulder. G. Berger is a BAEF fellow.  Email: guillaume.berger@colorado.edu.}
\thanks{This project has received funding from the European Research Council (ERC) under the European Union's Horizon 2020 research and innovation programme under grant agreement No 864017 - L2C.
R. Jungers is also supported by the Walloon Region and the Innoviris Foundation.
 }}

\maketitle
\begin{abstract}
This paper tackles state feedback control of switched linear systems under arbitrary switching.
We propose a data-driven control framework that allows to compute a stabilizing state feedback using only a finite set of observations of trajectories with quadratic and sum of squares (SOS) Lyapunov functions.
We do not require any knowledge on the dynamics or the switching signal, and as a consequence, we aim at solving \emph{uniform} stabilization problems in which the feedback is stabilizing for all possible switching sequences. In order to generalize the solution obtained from trajectories to the actual system, probabilistic guarantees on the obtained quadratic or SOS Lyapunov function are derived in the spirit of scenario optimization. For the quadratic Lyapunov technique, the generalization relies on a geometric analysis argument, while, for the SOS Lyapunov technique, we follow a sensitivity analysis argument. In order to deal with high-dimensional systems, we also develop parallelized schemes for both techniques. We show that, with some modifications, the data-driven quadratic Lyapunov technique can be extended to LQR control design. Finally, the proposed data-driven control framework is demonstrated on several numerical examples.
\end{abstract}


%%%%%%%%%%%%%%%%%%%%%%%%%%%%%%%%%%%%%%%%%%%%%%%%%%%%%%%%%%%%%%%%%%%%%%%%%%%%%%%%%%%%%%%%
\section{Introduction}
Switched systems are typical hybrid dynamical systems which consist of a number of dynamics modes and a switching rule selecting the current mode.
The jump from one mode to another often causes complicated hybrid behaviors resulting in significant challenges in stability analysis and control design, see, e.g.,  \cite{ART:LM99,BOO:L03,ART:LA09}. This paper focuses on stabilization and control of switched linear systems. 

Many control techniques have been proposed for switched systems depending on the assumptions on the switching rule.  In the case where the switching signal is the only control input,  a standard technique for achieving stabilization is to impose constraints on the switching sequence, e.g.,  dwell time  \cite{ART:M96,ART:GC06,INP:CCSGM10} and average dwell time \cite{INP:HM99,INP:ZHYM02}. More advanced techniques use state-dependent switching rules to achieve stabilization, see, e.g., \cite{ART:GC06,ART:FGJ15,ART:JM17} and the references therein. In \cite{ART:FJ14,ART:FAJ15}, necessary and sufficient conditions for stabilizability are also provided. For switched systems with affine control inputs, the stabilization problem becomes even more complicated due to extra freedom.
In \cite{ART:BM99,ART:LD04}, the (time) varying nature of the dynamics is considered as uncertainty and uniform state feedback stabilization laws are proposed for all possible switching sequences. When both the affine control input and the switching signal are accessible, exponential stabilization can be achieved for instance by using piecewise quadratic control Lyapunov functions which are essentially solutions to switched LQR problems \cite{ART:ZAHV09}. In the presence of state and input constraints, optimal control of switched linear systems is also addressed under the framework of model predictive control \cite{ART:OWD16,ART:VT19}. However, these stabilization methods all require a model of the underlying switched system.

While there exist hybrid system identification techniques \cite{BOO:LB18}, identification of state-space models of switching systems is in general cumbersome and computationally demanding. More specifically, identifying a switched linear system is NP-hard \cite{ART:L16} and most of currently available techniques as mentioned in \cite{BOO:LB18} rely on heuristics and lack of formal guarantees. In recent years, data-driven analysis and control under the framework of black-box systems has received a lot of attention, see, e.g., \cite{INP:KQHT16,ART:KBJT19,ART:SDDJMD20,ART:WJ20}. For instance, probabilistic stability guarantees are provided in \cite{ART:KBJT19,INP:RWJ21,INP:BJW21} for black-box switched linear systems, based on the observation of a finite set of trajectories. Let us also mention that, although data-driven techniques for controlling linear systems already exist (see, e.g., \cite{ART:VETC20}), they are not able to tackle switched systems.

In this paper, we address feedback control design for switched linear systems without knowing the model of the system or the switching signal.
As the switching is considered as a source of uncertainty, we need to design a uniform state feedback controller allowing to stabilize the system in the worst case, similar to \cite{ART:BM99,ART:LD04}.
To do this, we compute a state feedback controller and a common Lyapunov function for all the switching modes of the closed-loop system using a finite set of trajectories. We use both quadratic and sum of squares (SOS) Lyapunov functions, which lead to constrained nonlinear optimization problems with a large number of Lyapunov inequalities. For numerical tractability, we then design algorithms to solve these problems by making use of the underlying structure. For example, with quadratic Lyapunov functions, the biconvex Lyapunov inequalities allow to use alternating minimization where each iteration solves convex problems.

One major issue of this data-based feedback control design is that it is usually only valid for the regions where the data is sampled but may not stabilize the actual system in the whole space.
In order to formally describe the properties of the controller, we derive probabilistic stability guarantees in the spirit of scenario optimization \cite{ART:C10,ART:CGR18,ART:CG18}. In this context, one trajectory can be considered as a scenario and the stabilization problem formulated based on a set of trajectories is a sampled problem. As our problem is non-convex, the convex chance-constrained theorems in \cite{ART:C10} are not applicable. While chance-constrained theorems for nonlinear optimization problems also exist in \cite{ART:CGR18,ART:CG18}, their probabilistic bounds rely on the knowledge of the essential set (which is basically the set of irremovable constraints).
Identifying this set can be very expensive for general nonlinear problems, in particular for nonlinear semidefinite problems. Hence, the techniques in \cite{ART:CGR18,ART:CG18} are not suitable for our case which involves a large number of polynomial constraints and linear/bilinear matrix inequalities. Instead, probabilistic stability guarantees in this paper are derived relying on the notions of set covering and packing (see, e.g., \cite[Chapter 27]{BOO:SB14}) and geometric/sensitivity analysis of the underlying problem. Similar probabilistic guarantees are also developed in \cite{ART:KBJT19,INP:RWJ21,INP:BJW21} for autonomous systems. However, the probabilistic guarantees in \cite{ART:KBJT19,INP:RWJ21,INP:BJW21} all require the optimality of the obtained solution, while our techniques work with any feasible solution of the underlying optimization problem. This allows to parallelize our algorithms to substantially speed up the computations. Finally, we show that the proposed data-driven Lyapunov framework can be extended to switched LQR problems.

A preliminary version of this paper appears as a conference paper in \cite{INP:WBJ21} which only considers quadratic stabilization, and does not contain complete proofs. In this paper, we provide complete detailed proofs of all the results in \cite{INP:WBJ21}.  In particular, the present paper contains the first published proof of our main Theorem \ref{thm:gamma}. Furthermore, we present an extension to SOS stabilization which calls for a new technique for deriving probabilistic stability guarantees. Besides the stabilization problem, we also consider a switched LQR problem. In addition, to circumvent computational issues, we also present a parallelized scheme for both quadratic and SOS stabilization.

The rest of the paper is organized as follows. This section ends with the notation, followed by Section \ref{sec:pre} on the review of preliminary results on stability of switched linear systems and the formulation of the state feedback stabilization problem. Section \ref{sec:method} presents the proposed data-driven quadratic Lyapunov technique, including an alternating minimization algorithm, probabilistic stability analysis, and a parallelized scheme. In Section \ref{sec:sos}, the SOS Lyapunov technique is considered with a similar alternating minimization algorithm. In Section \ref{sec:LQR}, we extend the data-driven Lyapunov framework to switched LQR design. Numerical results are provided in Section \ref{sec:num}.

\textbf{Notation}.
The set of non-negative integers is denoted by $\mathbb{Z}^+$.
For a square matrix $Q$, $Q\succ(\succeq)~0$ means that $Q$ is symmetric and positive definite (semi-definite).
$\mathbb{S}$ and $\mathbb{B}$ are the unit sphere and the unit (closed) ball in $\mathbb{R}^n$ respectively.
$\mu(\cdot)$ denotes the uniform spherical measure on $\mathbb{S}$ with $\mu(\mathbb{S}) = 1$.
% For any matrix $P\succ 0$, we define  $\lVert x\rVert_P\coloneqq\sqrt{x^TPx}$, and we denote by $\lambda_{\max}(P)$ and $\lambda_{\min}(P)$ the largest and smallest eigen\-values of $P$ respectively.
For any square matrix $P$, $\textrm{tr}(P)$ denotes the trace of $P$.
For any symmetric matrix $P$, we denote by $\lambda_{\max}(P)$ and $\lambda_{\min}(P)$ the largest and smallest eigen\-values of $P$ respectively. For any matrix $P\succ 0$, let $\kappa(P)\coloneqq \lambda_{\max}(P)/\lambda_{\min}(P)$ be the condition number. For any $p\ge 1$, the $\ell_p$ norm of a vector $x\in \mathbb{R}^n$ is $\|x\|_p$ ($\|x\|$ is the $\ell_2$ norm by default) with $\|x\|_Q^2= x^\top Qx$ for any $Q\succeq 0$. Finally, given $x\in\mathbb{S}$ and $\theta\in [0,\pi/2]$, let $\SymCap(x,\theta)\coloneqq\{v\in \mathbb{S}: \lvert x^\top v\rvert \ge \cos(\theta)\}$ be the \emph{symmetric spherical cap} with direction $x$ and angle $\theta$.

%%%%%%%%%%%%%%%%%%%%%%%%%%%%%%%%%%%%%%%%%%%%%%%%%%%%%%%%%%%%%%%%%%%%%%%%%%%%%%%%%%%%%%%%%
\section{Preliminaries and problem statement}\label{sec:pre}
We consider the following switched linear system
\begin{align}\label{eqn:Asigma}
x(t+1) = A_{\sigma(t)}x(t) + Bu(t), \quad t\in \mathbb{Z}^+,
\end{align}
where $x(t)\in \mathbb{R}^n$ is the state vector,  $u(t)\in \mathbb{R}^m$ is the input and $\sigma: \mathbb{Z}^+\rightarrow \mathcal{M}\coloneqq\{1,2,\cdots,M\}$ is a time-dependent switching signal that indicates the current active mode of the system among $M$ possible modes in $\mathcal{A}\coloneqq\{A_1,A_2,\cdots, A_M\}$. In this paper, we consider the case in which the switching signal is changing arbitrarily and cannot be observed, i.e., the information on the switching signal is not available. Note that the input matrix $B$ is constant.  Our goal is to find a static linear state feedback stabilizing the system under arbitrary switching, that is, a feedback matrix $K\in \mathbb{R}^{m\times n}$ such that the closed-loop system below is stable for all switching signals
\begin{align}\label{eqn:AsigmaK}
x(t+1) = (A_{\sigma(t)}+BK) x(t), \quad t\in \mathbb{Z}^+.
\end{align}
For notational convenience, let $\mathcal{A}_K\coloneqq\{A_1+BK,A_2+BK,\cdots, A_M+BK\}$ for a given $K\in \mathbb{R}^{m\times n}$. The stability of System (\ref{eqn:AsigmaK}) under arbitrary switching can be described by the joint spectral radius (JSR) \cite{BOO:J09} of the matrix set $\mathcal{A}_K$ defined by 
\begin{align}\label{eqn:rhoAK}
\rho(\mathcal{A}_K)\coloneqq \lim\limits_{k\rightarrow \infty} \max\limits_{\pmb{\sigma}(k)\in \mathcal{M}^k}\|\bar{\pmb{A}}_{\pmb{\sigma}(k)} (K)\|^{1/k}
\end{align}
where $\pmb{\sigma}(k)\coloneqq\{\sigma(0),\sigma(1), \cdots, \sigma(k-1)\}$ and $\bar{\pmb{A}}_{\pmb{\sigma}(k)} (K) = (A_{\sigma(k-1)}+BK)\cdots (A_{\sigma(1)}+BK) (A_{\sigma(0)}+BK)$. System (\ref{eqn:AsigmaK}) is asymptotically stable when $\rho(\mathcal{A}_K)<1$. Hence, the state feedback stabilization problem for System (\ref{eqn:Asigma}) amounts to finding a $K\in \mathbb{R}^{m\times n}$ such that $\rho(\mathcal{A}_K)<1$. However, the computation of the JSR of a set of matrices is known to be a very challenging problem in general, let alone its optimization in the context of control design. For this reason, we use tractable upper bounds on the JSR, providing sufficient conditions for stability or stabilization, see \cite{BOO:J09}. The following proposition provides a sufficient condition based on a common quadratic Lyapunov function which can be computed via semidefinite programming (SDP) \cite{BOO:BV04}.

\begin{proposition}[{\cite[Prop. 2.8]{BOO:J09}}]\label{prop:jsrq}
Consider the closed-loop matrices $\mathcal{A}_K$ for some state feedback $K\in \mathbb{R}^{m\times n}$. If there exist $\gamma\ge 0$ and $P \succ 0$ such that $A^\top PA \preceq \gamma^2 P$, $\forall\,A\in \mathcal{A}_K$, then $\rho(\mathcal{A}_K) \le \gamma$.
\end{proposition}

From this proposition, we formulate the following nonlinear semidefinite optimization problem for stabilization of switched linear systems:
\begin{subequations}\label{eqn:PK}
\begin{align}
\gamma^*\coloneqq&\min_{\gamma\ge 0, P, K} \gamma\\
\textrm{s.t.} \quad &(A+BK)^\top P (A+BK) \preceq \gamma^2 P,\: \forall\,A\in \mathcal{A}\\
& P \succ 0.
\end{align}
\end{subequations}
Using the Schur complement formula \cite[Theorem 1.12]{BOO:Z06} with $S = P^{-1}$ and $Y = KS$, the nonlinear constraints in (\ref{eqn:PK}) can be converted into linear matrix inequalities (LMI) when $\gamma$ is fixed:
\begin{subequations}\label{eqn:QY}
\begin{align}
&\min_{\gamma\ge 0, S, Y} \gamma\\
\textrm{s.t.} \quad  & \!\left(\!\!\begin{array}{cc}
\gamma^2 S & SA^\top + Y^\top B^\top\\
AS + BY & S
\end{array}\!\!\right) \succeq 0,\: \forall\, A\in \mathcal{A}\\
& S \succ 0.
\end{align}
\end{subequations}
Such a transformation is widely used in stability analysis and control design, see, e.g., \cite{ART:VB00}. When the matrices $\mathcal{A}$ are known, Problem (\ref{eqn:QY}) can be efficiently solved via semidefinite programming and bisection on $\gamma$. 

In this paper, we aim to solve the stabilization problem of switched linear systems when the matrices $\mathcal{A}$ are unknown and only a finite set of trajectories of the system are observed. Such systems are called black-box switched linear systems as in \cite{ART:KBJT19}. To this end, we reformulate Problem (\ref{eqn:PK}) as a problem with an infinite number of constraints below:
\begin{subequations}\label{eqn:PKS}
\begin{align}
&\min_{\gamma\ge 0, P, K} \gamma\\
\textrm{s.t.} \quad  & (Ax+BKx)^\top P (Ax+BKx) \le \gamma^2 x^\top Px, \nonumber\\
&\qquad  \forall\,A\in \mathcal{A},\: \forall\,x\in \mathbb{R}^n \label{eqn:AxBKxPgammaR}\\
& P \succ 0.
\end{align}
\end{subequations}
By homogeneity, one can restrict the constraints in (\ref{eqn:AxBKxPgammaR}) to the set of points in the unit sphere $\mathbb{S}$ instead of the whole space $\mathbb{R}^n$. As we will show later, the formulation in (\ref{eqn:PKS}) allows us to develop a model-free and data-based control design. For that, we make the following assumption about the observations available.

\begin{assumption}
The state $x(t)$ can be fully observed for all $t\in \mathbb{Z}^+$, the input matrix $B$ is time-invariant and known, and the number of modes (or an upper bound) is available.
\end{assumption}

The assumption that $B$ is time-invariant is not restrictive in many applications, for instance, when the switching only occurs in some parameters of the dynamics. Such an assumption is often made in the literature, see, e.g., \cite{ART:BM99,ART:SRD19}.

%%%%%%%%%%%%%%%%%%%%%%%%%%%%%%%%%%%%%%%%%%%%%%%%%%%%%%%%%%%%%%%%%%%%%%%%%%%%%%%%%%%%%%%%%%%%%%%%%%%%%%
\section{State feedback stabilization}\label{sec:method}
This section presents our model-free quadratic Lyapunov technique for stabilizing black-box switched linear systems. We first formulate a sample-based stabilization problem, which consists of a set of biconvex constraints. Then, to solve this problem, we present an alternating minimization algorithm that generates feasible iterates. With the concepts of covering/packing numbers, probabilistic guarantees on the obtained solution are then provided using geometric analysis. Finally, we also show that the algorithm can be parallelized to speed up the computation.

\subsection{Sampled stabilization problem} \label{sec:samplestable}
For the model-free design, we sample a finite number of pairs of the initial state and switching mode. More precisely, we randomly and uniformly generate $N$ initial states on $\mathbb{S}$ and $N$ modes in $\mathcal{M}$, which are denoted by $\omega_N\coloneqq\{(x_i,\sigma_i)\in \mathbb{S} \times \mathcal{M}: i=1,2,\cdots,N\}$.  From this random sampling, we observe the trajectories of  the open-loop system of System (\ref{eqn:Asigma}) with $u=0$ and obtain the observed data set $\{(x_i,A_{\sigma_i}x_i): i=1,2,\cdots,N\}$, where $A_{\sigma_i}x_i$ is the successor of the initial state $x_i$ with respect to mode $\sigma_i$. Note that although $A_{\sigma_i}x_i$ depends on $\sigma_i$, $\sigma_i$ is not directly observed. 

For the given sample set $\omega_N$, we define the following \emph{sampled} problem:
\begin{subequations}\label{eqn:gammaomegaN}
\begin{align}
&\min_{\gamma\ge 0, P, K} \gamma\\
\textrm{s.t.} \quad  & (A_\sigma x+BKx)^\top P (A_\sigma x+BKx) \le \gamma^2 x^\top Px, \nonumber\\
&\qquad  \forall\,(x,\sigma) \in \omega_N \label{eqn:gammaomegaNxp}\\
& P \succ 0.
\end{align}
\end{subequations}
From the homogeneity property of (\ref{eqn:gammaomegaNxp}), the inequality $P\succ 0$ can be replaced by $P \succeq I$ due to the fact that the feasibility of $P$ implies the feasibility of $P/\lambda_{\min}(P)$. When the size of $\omega_N$ is small, the sampled problem (\ref{eqn:gammaomegaN}) can be solved using polynomial optimization toolboxes \cite{ART:L01,ART:HL03,sostools}. As $P$ is invertible, the Schur complement formula \cite[Theorem 1.12]{BOO:Z06} can be applied to (\ref{eqn:gammaomegaNxp}) using the reformulation below
\begin{align}
	(A_\sigma x+BKx)^\top P P^{-1} P (A_\sigma x+BKx) \le \gamma^2 x^\top Px, \nonumber\\
	\forall\,(x,\sigma) \in \omega_N.
\end{align}
We can then convert the constraints in (\ref{eqn:gammaomegaNxp}) into a set of bilinear matrix inequalities (BMI) and reformulate Problem (\ref{eqn:gammaomegaN}) as the following BMI problem
\begin{subequations}\label{eqn:PKIomega}
\begin{align}
&\min_{\gamma\ge 0, P \succeq I, K} \gamma\\
\textrm{s.t.} \quad  &\!\left(\!\!\begin{array}{cc}
\gamma^2 x^\top P x  & (A_\sigma x+BKx)^\top P\\
P (A_\sigma x+BKx) & P
\end{array}\!\!\right) \succeq 0, \nonumber\\
&\qquad  \forall\,(x,\sigma) \in \omega_N.  \label{eqn:PKIomegaxp}
\end{align}
\end{subequations}
This reformulation allows us to solve Problem (\ref{eqn:gammaomegaN}) using BMI solvers \cite{ART:VB00,INP:HLKS05,INP:KZM18}.

%The advantage of this reformulation is that the trilinear constraints involving $P$ and $K$ in Problem (\ref{eqn:gammaomegaN}) are decoupled into bilinear constraints, which allows to use an alternating algorithm as shown below. 

\subsection{An alternating algorithm}\label{sec:alternating}
As the size of $\omega_N$ increases, it becomes numerically intractable to find a (local) optimum of (\ref{eqn:gammaomegaN}) or (\ref{eqn:PKIomega}) using the aforementioned polynomial or BMI solvers. Since we do not seek to have optimality, we can use a less costly approach described below. We propose an alternating minimization algorithm between $P$ and $K$ for its numerical tractability and simple implementation. As we will show later, this alternating algorithm also enables us to parallelize the computation. Given a fixed $K$, we also define:
\begin{subequations}\label{eqn:Pgamma}
\begin{align}
\bar{\mathcal{P}}(\omega_N;K)\coloneqq &\min_{\gamma\ge 0,  P \succeq I} \gamma\\
\textrm{s.t.} \quad  & (A_\sigma x+BKx)^\top P (A_\sigma x+BKx) \le \gamma^2 x^\top Px \nonumber\\
&\qquad  \forall\,(x,\sigma) \in \omega_N \label{eqn:AsigmaxBKxP}
\end{align}
\end{subequations}
For fixed values of $\gamma$, the constraints (\ref{eqn:AsigmaxBKxP}) reduce to LMIs, so that Problem (\ref{eqn:Pgamma}) can be solved efficiently using SDP solvers \cite{BOO:BV04} and bisection on $\gamma$, with the solution of (\ref{eqn:Kgamma}) being the initial feasible guess. Given a fixed $P$, we define:
\begin{subequations}\label{eqn:Kgamma}
	\begin{align}
		\hat{\mathcal{P}}(\omega_N;P)\coloneqq &\min_{\gamma\ge 0, K} \gamma\\
		\textrm{s.t.} \quad  & (A_\sigma x+BKx)^\top P (A_\sigma x+BKx) \le \gamma^2 x^\top Px \nonumber\\
		&\qquad  \forall\,(x,\sigma) \in \omega_N
	\end{align}
\end{subequations}
Problem (\ref{eqn:Kgamma}) is a second-order cone program that can be solved by well-documented convex solvers, like interior point methods \cite{BOO:BV04}. The overall procedure is summarized in Algorithm \ref{algo:dataK}. Note that this alternating algorithm always terminates though it does not necessarily converge to a (local) optimum of Problem (\ref{eqn:PKIomega}).

%\begin{align*}
%& (A_\sigma x+BK^*(\omega_N)x)^\top P^*(\omega_N) (A_\sigma x+BK^*(\omega_N)x) \\
%\le &(\gamma^*(\omega_N))^2 x^\top P^*(\omega_N)x, \quad  \forall\,(x,\sigma) \in \omega_N\\
%\end{align*}
%
%\begin{subequations}
%\begin{align*}
%&\min_{\gamma\ge 0, P\succeq I, K} \gamma\\
%\textrm{s.t.} \quad  & (A_\sigma x+BKx)^\top P (A_\sigma x+BKx) \le \gamma^2 x^\top Px, ~  \forall\,(x,\sigma) \in \omega_N 
%\end{align*}
%\end{subequations}

\begin{algorithm}[h]
\caption{Alternating minimization for quadratic stabilization}
\begin{algorithmic}[1]
\renewcommand{\algorithmicrequire}{\textbf{Input:}}
\renewcommand{\algorithmicensure}{\textbf{Output:}}
\REQUIRE $\{(x_i,A_{\sigma_i}x_i)\} _{i=1}^N$, $B$  and some tolerance $\epsilon_{tol}>0$
\ENSURE $\gamma(\omega_N),P(\omega_N)$, and $K(\omega_N)$\\
\textit{Initialization}: Let $k\leftarrow 0$, $K_k \leftarrow 0$, $P_k \leftarrow I$, and $\gamma_k \leftarrow \max_{(x,\sigma)\in \omega_N} \frac{ \|A_\sigma x\|}{\|x\|}$; 
\STATE {\new Obtain $P_{k+1}$ from (\ref{eqn:Pgamma}) with $K=K_k$ via bisection on $\gamma$ starting from $\max_{(x,\sigma)\in \omega_N} \frac{ \|(A_{\sigma} x+BK_kx)\|_{P_k}}{\|x\|_{P_k}}$;}
\STATE Obtain $K_{k+1}$ and $\gamma_{k+1}$ from (\ref{eqn:Kgamma}) with $P = P_{k+1}$;
%\STATE Obtain $P_{k+1}$ and $\gamma_{k+1}$ from (\ref{eqn:Pgamma}) with $K = K_{k+1}$ via bisection on $\gamma$ starting from $\max_{i} \frac{ \|(A_{\sigma_i} x_i+BK_{k+1}x_i)\|_{P_k}}{\|x_i\|_{P_k}}$;
\IF {$\|\gamma_{k+1}-\gamma_k\|< \epsilon_{tol}$}
\STATE $\gamma(\omega_N)\leftarrow \gamma_{k+1}$, $P(\omega_N)\leftarrow P_{k+1}$, $K(\omega_N)\leftarrow K_{k+1}$;
\STATE Terminate;
\ELSE
\STATE Let $k\leftarrow k+1$ and go to Step 2.
\ENDIF
\end{algorithmic}
\label{algo:dataK}
\end{algorithm}

\subsection{Probabilistic stability guarantees}\label{sec:prob}
We now derive formal stability guarantees on the solution obtained from Algorithm \ref{algo:dataK}. Let us first introduce a few definitions and notation. For any $\theta\in[0,\pi/2]$ and any $x\in\mathbb{S}$, we let $\delta(\theta)$ denote the relative area of the \emph{symmetric spherical cap} $\SymCap(x,\theta)$, given by $\{v\in \mathbb{S}: \lvert x^\top v\rvert \ge \cos(\theta)\}$. From \cite{ART:Ls11}, it holds that
\begin{align}\label{eqn:deltatheta}
	\delta(\theta)=\mathcal{I}(\sin^2(\theta);\frac{n-1}{2},\frac{1}{2})
\end{align}
where $\mathcal{I}(x;a,b)$ is the regularized incomplete beta function defined as
\begin{align}
	\mathcal{I}(x;a,b) \coloneqq \frac{\int_0^x t^{a-1}(1-t)^{b-1}dt}{\int_0^1 t^{a-1}(1-t)^{b-1}dt};
\end{align}
The function $\delta$ is strictly increasing with $\theta$, and thus we can define its inverse, denoted by $\delta^{-1}$, see also Figure~\ref{fig:delta_theta} for an illustration. Let $\delta_v(\theta)$ denote the relative volume of the portion of the unit ball $\mathbb{B}$ enclosed by $\SymCap(x,\theta)$ and the hyperplanes $\lvert x^\top v\rvert = \cos(\theta)$, expressed as $\{z\in \mathbb{B}: |x^\top z|\ge \cos(\theta)\}$. We recall again from \cite{ART:Ls11} that $\delta_v(\theta)$ can be given by
\begin{align}\label{eqn:deltavtheta}
	\delta_v(\theta) = \mathcal{I}(\sin^2(\theta);\frac{n+1}{2},\frac{1}{2}), \quad \forall \theta\in[0,\pi/2].
\end{align}
Similarly, the function $\delta_v$ is also strictly increasing with $\theta$ and its inverse is denoted by $\delta_v^{-1}$.

\begin{figure}[h]
\centering
\includegraphics[width=0.8\linewidth]{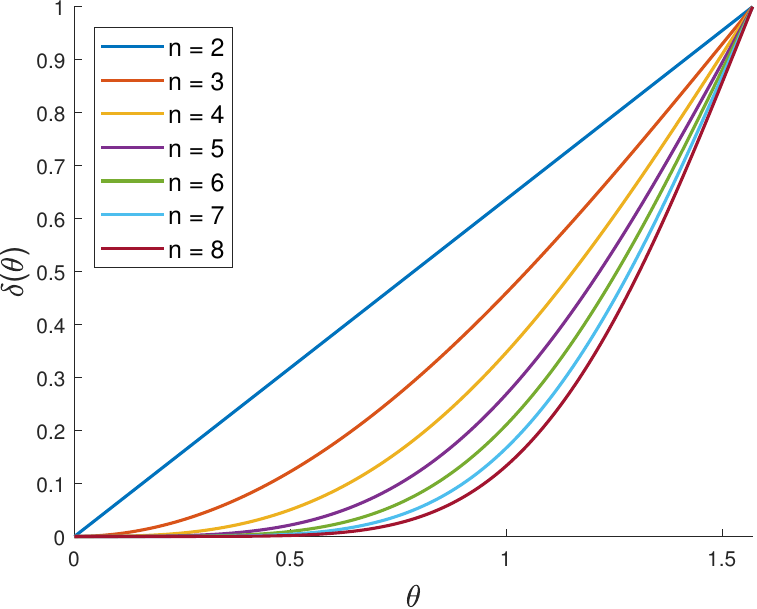}
\caption{ Measure $\mu$ of the symmetric spherical cap $\SymCap(x,\theta)$ in $\mathbb{R}^n$ for different values of $n$.}
\label{fig:delta_theta}
\end{figure}

% \begin{align}
% \delta(\epsilon) &\coloneqq \sqrt{1-\mathcal{I}^{-1}(2\epsilon;\frac{n-1}{2},\frac{1}{2})}, \label{eqn:deltaep}\\
% \theta(\epsilon) &\coloneqq \cos^{-1}(\delta(\epsilon)), \label{eqn:thetaep} 
% \end{align}
% where $\mathcal{I}(x;a,b)$ is the regularized incomplete beta function (see, e.g., \cite{ART:KBJT19}) defined as
% \begin{align}
% \mathcal{I}(x;a,b) \coloneqq \frac{\int_0^x t^{a-1}(1-t)^{b-1}dt}{\int_0^1 t^{a-1}(1-t)^{b-1}dt}.
% \end{align}
% From a geometric point of view, $\delta(\epsilon) $ can be considered as the distance from the origin to the base of a spherical cap of measure $\epsilon$ and $\theta(\epsilon)$ is the polar angle. 

Let us also recall the notions of covering and packing numbers, see Chapter 27 of \cite{BOO:SB14} for details. We adapt the classical definitions to the unit sphere.
\begin{definition}
Given $\epsilon\in (0,1)$, a set $Z \subset \mathbb{S}$ is called an \emph{$\epsilon$-covering} of $\mathbb{S}$ if, for any $x\in \mathbb{S}$, there exists $z\in Z$ such that $|z^\top x| \ge \cos(\theta)$ where $\theta=\delta^{-1}(\epsilon)$.
The \emph{covering number} $\mathcal{N}_c(\epsilon)$ is the minimal cardinality of an $\epsilon$-covering of $\mathbb{S}$.
\end{definition}

\begin{definition}
Given $\epsilon\in (0,1)$, a set $Z \subset \mathbb{S}$ is called an \emph{$\epsilon$-packing} of $\mathbb{S}$ if, for any two $z,v\in Z$,  $|z^\top v| < \cos(\theta)$ where $\theta=\delta^{-1}(\epsilon)$. The \emph{packing number} $\mathcal{N}_p(\epsilon)$ is the maximal cardinality of an $\epsilon$-packing of $\mathbb{S}$.
\end{definition}
With these definitions, we also adapt fundamental results on set covering and packing (see \cite[Chapter 27]{BOO:SB14}) to the unit sphere, as stated below. 
\begin{lemma}\label{lem:covering}
For any $\epsilon\in (0,1)$,
\begin{align}
\mathcal{N}_c(\epsilon) \le \mathcal{N}_p(\epsilon) \le \frac{1}{\delta(\frac12\delta^{-1}(\epsilon))}.
\end{align}
\end{lemma}

\begin{proof}
The first inequality follows from the fact that any $\epsilon$-packing with maximal cardinality is also an $\epsilon$-covering.
To prove the second inequality, let $Z$ be the $\epsilon$-packing with the maximal cardinality.
Let $\theta=\delta^{-1}(\epsilon)$.
From the definition of an $\epsilon$-packing, the spherical caps $\{\SymCap(z,\theta/2)\}_{z\in Z}$ are disjoint.
Hence, $\sum_{z\in Z} \mu(\SymCap(z,\theta/2))\leq 1$, which leads to the second inequality.
\end{proof}

\begin{remark}
% The definitions above are similar to those in \cite{ART:WJ21b}, except that here we take into account the symmetry property of quadratic functions.
The definitions above are similar to those in \cite{ART:WJ21b}, except that we consider \emph{symmetric spherical caps} in the form of $\{v\in \mathbb{S}: \lvert x^\top v\rvert \ge \cos(\theta)\}$ given any $\theta\in[0,\pi/2]$ and any $x\in\mathbb{S}$, to take into account the symmetry of the problem.
\end{remark}

We then extend the definition of $\epsilon$-covering to the joint set $\mathbb{S}\times \mathcal{M}$ as follows.

\begin{definition}\label{def:joint}
Given $\epsilon\in (0,1)$, a set $\omega \subset \mathbb{S}\times \mathcal{M}$ is called an \emph{$\epsilon$-covering} of $\mathbb{S}\times \mathcal{M}$ if, for any $(x,\sigma)\in \mathbb{S}\times \mathcal{M}$, there exists $(z,\sigma)\in \omega$ such that $|z^\top x| \ge \cos(\theta)$ where $\theta=\delta^{-1}(\epsilon)$.
\end{definition}

The following lemma shows probabilistic properties of the sample set $\omega_N$, which will be needed below in order to achieve formal guarantees on the controller.

\begin{lemma}\label{lem:omegaNcover}
Given $N\in \mathbb{Z}^+$, let $\omega_N=\{(x_i,\sigma_i)\}_{i=1}^N$ be independent and identically distributed (i.i.d) with respect to the uniform distribution $\mathbb{P}$ over $\mathbb{S}\times \mathcal{M}$. Then, given any $\epsilon\in (0,1)$, with probability no smaller than $1-\mathcal{B}(\epsilon;N)$, $\omega_N$ is an $\epsilon$-covering of $\mathbb{S}\times \mathcal{M}$, where 
\begin{align}\label{eqn:mathcalB}
\mathcal{B}(\epsilon;N) &\coloneqq \frac{M\bigg(1-\dfrac{\delta\left(\frac12\delta^{-1}(\epsilon)\right)}{M}\bigg)^N}{\delta(\frac14\delta^{-1}(\epsilon))}.
\end{align}
\end{lemma}
\begin{proof}
Consider a maximal $\epsilon'$-packing $Z$ of $\mathbb{S}$ with $\epsilon'=\delta(\frac12\delta^{-1}(\epsilon))$ and let $\theta=\delta^{-1}(\epsilon')=\frac12\delta^{-1}(\epsilon)$.
From the proof of Lemma \ref{lem:covering}, $\{\SymCap(z,\theta)\}_{z\in Z}$ covers $\mathbb{S}$.
Suppose $\omega_N$ is sampled randomly according to the uniform distribution, then the probability that each set in $\{\SymCap(z,\theta)\}_{z\in Z}$ contains $M$ points with $M$ different modes is no smaller than $1-\mathcal{N}_p(\epsilon')M(1-\frac{\epsilon'}{M})^N \geq1-\mathcal{B}(\epsilon;N)$.
When this happens, for any $(x,\sigma)\in \mathbb{S} \times \mathcal{M}$, there exists a pair $(z,\sigma) \in \omega_N$ such that $|x^\top z| \ge \cos(2\theta) $.
This completes the proof.
\end{proof}

We are now able to present a key result of this section, which provides a stability certificate from the solution of the sampled problem (\ref{eqn:gammaomegaN}) assuming sufficient covering by the sample set.
\begin{theorem}\label{thm:gamma}
Given a sample set $\omega_N \subset \mathbb{S}\times \mathcal{M}$, consider Problem (\ref{eqn:gammaomegaN}). Let $(\gamma,P,K)$ be a feasible solution to Problem (\ref{eqn:gammaomegaN}). Suppose that $\omega_N$ is an $\epsilon$-covering of $\mathbb{S}\times \mathcal{M}$ for some $\epsilon \in (0,1)$. Then,
\begin{align}\label{eqn:gammastarover}
\rho(\mathcal{A}_{K}) \le \frac{\gamma}{\max\{\varphi_P(\epsilon),\psi_P(\epsilon)\}}
\end{align}
where $\rho(\mathcal{A}_{K})$ is defined in (\ref{eqn:rhoAK}) and 
\begin{align}
	\varphi_P(\epsilon) &\coloneqq 1-\kappa(P)(1-\cos(\delta^{-1}(\epsilon))) \label{eqn:varphiPepsilon}\\
	\psi_P(\epsilon) & \coloneqq \cos\left(\delta_v^{-1}(1-\sqrt{\frac{\det(P)}{\lambda_{\max}(P)^n}}\cos(\delta^{-1}(\epsilon))^n)\right) \label{eqn:psiPepsilon}
\end{align}
with $\delta(\cdot)$ and $\delta_v(\cdot)$ being given in (\ref{eqn:deltatheta}) and \eqref{eqn:deltavtheta} respectively.
\end{theorem}

\begin{proof}
We drop the subscript $N$ in $\omega_N$ in the proof for convenience. Let $P=L^\top L$ be the Cholesky decomposition of $P$, and let 
\begin{align}\label{eqn:omegaNtilde}
\tilde{\omega} \coloneqq \Big\{\Big(\frac{Lz}{\|Lz\|}, \sigma\Big): (z,\sigma) \in \omega \Big\} \subset \mathbb{S}\times \mathcal{M}.
\end{align}
We first show that $\rho(\mathcal{A}_{K}) \le \frac{\gamma}{\varphi_P(\epsilon)}$. The proof is divided into two steps.

\emph{Step~1:}
We show that if $\omega$ is an $\epsilon$-covering of $\mathbb{S}\times \mathcal{M}$, then $\tilde{\omega}$ is an $\tilde{\epsilon}$-covering of $\mathbb{S}\times \mathcal{M}$ for some $\tilde{\epsilon}>0$ defined below, that is, for any $(\tilde{x},\sigma) \in \mathbb{S}\times \mathcal{M}$, we want to show that there exists $(\tilde{z},\sigma)\in \tilde{\omega}$ such that $|\tilde{z}^\top \tilde{x}| \ge\cos(\tilde\theta)$ where $\tilde\theta=\delta^{-1}(\tilde{\epsilon})$.
Note that any $\tilde{x} \in \mathbb{S}$ can be uniquely expressed as $\tilde{x} =  Lx/ \|Lx\|$ for some $x\in \mathbb{S}$.
Let $\tilde{x}=Lx/\|Lx\|\in\mathbb{S}$.
Since $\omega$ is an $\epsilon$-covering of $\mathbb{S}\times \mathcal{M}$, from the definition, there exists $(z,\sigma)\in \omega$ such that $\lvert x^\top z\rvert \ge\cos(\theta)$ where $\theta=\delta^{-1}(\epsilon)$, which implies that $\|x-z\|\le \sqrt{2-2\cos(\theta)}$ or $\|x+z\|\le \sqrt{2-2\cos(\theta)}$.
Now, let us look at the value $\lvert\tilde{x}^\top \tilde{z}\rvert$ where $\tilde{z}=Lz/\|Lz\|\in\tilde\omega$. 
Without loss of generality, we consider the case that $\|x-z\|\le \sqrt{2-2\cos(\theta)}$.
Hence, 
\begin{align}
\frac{|(Lx)^\top Lz|}{\|Lx\| \|Lz\|} &= \frac{|x^\top Pz|}{\|Lx\| \|Lz\|} \nonumber \\
&= \frac{|x^\top Px + z^\top P z - (x-z)^\top P (x-z)|}{2\|Lx\| \|Lz\|} \nonumber \\
& \ge \frac{x^\top Px + z^\top P z - |(x-z)^\top P (x-z)|}{2 \sqrt{x^\top Px} \sqrt{z^\top P z}} \nonumber \\
& \ge 1- \frac{|(x-z)^\top P (x-z)|}{2 \sqrt{x^\top Px} \sqrt{z^\top P z}} \label{eqn:sqrtxPxzPZ}\\
& \ge 1 - \frac{\|x-z\|^2  \lambda_{\max}(P)}{2 \lambda_{\min}(P)} \label{eqn:sqrtxzPxz} \\
& \ge  1 - \kappa(P) (1-\cos(\theta)) = \varphi_P(\epsilon)
\end{align}
where (\ref{eqn:sqrtxPxzPZ}) follows from the fact that $x^\top Px + z^\top P z \ge 2\sqrt{x^\top Px} \sqrt{z^\top Pz}$ and (\ref{eqn:sqrtxzPxz}) is a direct consequence of the inequality $\lambda_{\min}(P) I \preceq P \preceq \lambda_{\max}(P) I $. Hence, $\tilde{\omega}_N$ is  a $\tilde{\epsilon}$-covering of $\mathbb{S}\times \mathcal{M}$ with $\tilde\epsilon=\delta(\tilde\theta)$ and $\cos(\tilde\theta) = 1 - \kappa(P) (1-\cos(\theta))$.

\emph{Step~2:}
Now, we are in a position to show $\rho(\mathcal{A}_{K}) \le \frac{\gamma}{\varphi(P,\epsilon)}$. Let us define:
\begin{align}\label{eqn:tildeomegasigma}
\tilde{\omega}^{\sigma}\coloneqq\{x: (x,\sigma) \in \tilde{\omega} \},\: \forall\, \sigma \in \mathcal{M}.
\end{align}
From Definition \ref{def:joint}, we know that $\tilde{\omega}^{\sigma}$ is a $\tilde{\epsilon}$-covering of $\mathbb{S}$ for all $\sigma \in \mathcal{M}$. This implies that $\cos(\tilde\theta) \mathbb{B} \subseteq \textrm{conv} (\pm \tilde{\omega}^{\sigma})$ for all $\sigma \in \mathcal{M}$, where $\pm \tilde{\omega}^{\sigma}$ denotes the union of $\tilde{\omega}^{\sigma}$ and $-\tilde{\omega}^{\sigma}$, i.e., $\tilde{\omega}^{\sigma}\cup -\tilde{\omega}^{\sigma}$.
Hence,
\begin{align}
\cos(\tilde\theta) \mathbb{B} \subseteq \bigcap_{\sigma \in \mathcal{M}} \textrm{conv} (\pm\tilde{\omega}^{\sigma}). \label{eqn:deltaB}
\end{align}
Let $\tilde{A}_\sigma \coloneqq LA_{\sigma}L^{-1}+LBK L^{-1}$ for all $\sigma \in \mathcal{M}$. It holds that for all $\sigma\in\mathcal{M}$ and $z \in \pm\tilde{\omega}^{\sigma} $, $\|\tilde{A}_\sigma z\| \le \gamma \|z\|$.
This, together with (\ref{eqn:deltaB}), implies that, $\forall\, \sigma \in \mathcal{M}$,
\begin{align*}
\frac{\cos(\tilde\theta)}{\gamma} \tilde{A}_\sigma \mathbb{B} \subseteq \frac{1}{\gamma} \tilde{A}_\sigma \textrm{conv} (\pm\tilde{\omega}^{\sigma}) \subseteq \frac{1}{\gamma} \textrm{conv} (\pm\tilde{A}_\sigma  \tilde{\omega}^{\sigma}) \subseteq \mathbb{B}.
\end{align*}
As a consequence, we obtain that, $\forall \sigma \in \mathcal{M}$,
\begin{align}\label{eqn:LMItilde}
(A_{\sigma}+BK)^\top P (A_{\sigma}+BK) \preceq \left(\frac{\gamma}{\cos(\tilde\theta)}\right)^2 P.
\end{align}
Finally, by combining \eqref{eqn:LMItilde} with Proposition \ref{prop:jsrq}, we get that $\frac{\gamma}{\cos(\tilde\theta)} = \frac{\gamma}{\varphi_P(\epsilon)}$ is an upper bound on $\rho(\mathcal{A}_K)$.

We then prove that $\rho(\mathcal{A}_{K}) \le \frac{\gamma}{\psi_P(\epsilon)}$. Similarly, let $\omega^\sigma \coloneqq \{x: (x,\sigma)\in \omega\}, \forall \sigma \in \mathcal{M}$. We consider an arbitrary $\sigma \in \mathcal{M}$. Since $\omega^\sigma$ is an $\epsilon$-covering of $\mathbb{S}$, $\cos(\delta^{-1}(\epsilon)) \mathbb{B} \subseteq \textrm{conv} (\pm\omega^{\sigma})$. Hence, 
\begin{align}\label{eqn:lambdaconvB}
	\lambda(\textrm{conv} (\pm\omega^{\sigma})) \ge \cos(\delta^{-1}(\epsilon))^n \lambda(\mathbb{B}),
\end{align}
where $\lambda(\cdot)$ denotes the Lebesgue measure. Note that $\tilde{\omega}^{\sigma}$ as defined in \eqref{eqn:tildeomegasigma} can be expressed as $\tilde{\omega}^{\sigma} = \{\frac{Lz}{\|Lz\|}: z\in \omega^{\sigma}\}$, which leads to the following relation 
\begin{align}\label{eqn:convL}
	\textrm{conv} (\pm\tilde{\omega}^{\sigma}) \supseteq \frac{L}{\sqrt{\lambda_{\max}(P)}} \textrm{conv} (\pm\omega^{\sigma}).
\end{align}
Combining \eqref{eqn:lambdaconvB} and \eqref{eqn:convL} yields 
\begin{align}
\frac{\lambda(\textrm{conv} (\pm\tilde{\omega}^{\sigma})}{\lambda(\mathbb{B})} \ge \sqrt{\frac{\det(P)}{\lambda_{\max}(P)^n}} \cos(\delta^{-1}(\epsilon))^n,
\end{align}
which implies that
\begin{align}\label{eqn:lambdaBconv}
	\frac{\lambda(\mathbb{B}\setminus\textrm{conv} (\pm\tilde{\omega}^{\sigma}))}{\lambda(\mathbb{B})} \le 1- \sqrt{\frac{\det(P)}{\lambda_{\max}(P)^n}} \cos(\delta^{-1}(\epsilon))^n.
\end{align}
We claim that the distance from $\partial\left( \textrm{conv} (\pm\tilde{\omega}^{\sigma}) \right)$ to the origin is  bounded from below by $\psi_P(\epsilon) = \cos\left(\delta_v^{-1}(1-\sqrt{\frac{\det(P)}{\lambda_{\max}(P)^n}}\cos(\delta^{-1}(\epsilon))^n)\right)$. We go by contradiction. Suppose there exists a point $x\in \partial\left( \textrm{conv} (\pm\tilde{\omega}^{\sigma}) \right)$ such that $\|x\| < \psi_P(\epsilon)$. Then, there exists a hyperplane $h^\top x =1$ such that $\frac{1}{\|h\|} < \psi_P(\epsilon)$ and $h^\top x \le 1$ for any $x\in  \textrm{conv} (\pm\tilde{\omega}^{\sigma})$. By symmetry, it also holds that $-h^\top x \le 1$ for any $x\in  \textrm{conv} (\pm\tilde{\omega}^{\sigma})$. Let $\tilde{B}$ denote the set $\{x\in \mathbb{B}: h^\top x \ge 1\} \cup \{x\in \mathbb{B}: -h^\top x \ge 1\}$. By construction, $\tilde{B} \subseteq \mathbb{B}\setminus\textrm{conv} (\pm\tilde{\omega}^{\sigma})$, which means that $\lambda(\tilde{B}) \le (1- \sqrt{\frac{\det(P)}{\lambda_{\max}(P)^n}} \cos(\delta^{-1}(\epsilon))^n)\lambda(\mathbb{B})$ from \eqref{eqn:lambdaBconv}. We recall from \cite{ART:Ls11} that the volume of $\tilde{B}$ is $\mathcal{I}(1-\frac{1}{\|h\|^2};\frac{n+1}{2},\frac{1}{2}) \lambda(\mathbb{B})$, which, from the condition that $\frac{1}{\|h\|} < \psi_P(\epsilon)$, implies that $\lambda(\tilde{B}) > \mathcal{I}(1-\psi_P(\epsilon)^2;\frac{n+1}{2},\frac{1}{2}) \lambda(\mathbb{B}) =  (1- \sqrt{\frac{\det(P)}{\lambda_{\max}(P)^n}} \cos(\delta^{-1}(\epsilon))^n)\lambda(\mathbb{B})$, where the equality is by the definition of $\psi_P(\epsilon)$ in \eqref{eqn:psiPepsilon}. This leads to a contradiction. Therefore, we conclude that $\psi_P(\epsilon) \mathbb{B} \subseteq \textrm{conv} (\pm\tilde{\omega}^{\sigma})$. As $\sigma$ is chosen arbitrarily, this implies that
\begin{align*}
\frac{\psi_P(\epsilon)}{\gamma} \tilde{A}_\sigma \mathbb{B} \subseteq \frac{1}{\gamma} \tilde{A}_\sigma \textrm{conv} (\pm\tilde{\omega}^{\sigma}) \subseteq \frac{1}{\gamma} \textrm{conv} (\pm\tilde{A}_\sigma  \tilde{\omega}^{\sigma}) \subseteq \mathbb{B}.
\end{align*}
Thus, $\forall \sigma \in \mathcal{M}$,
\begin{align}\label{eqn:ApsiP}
	(A_{\sigma}+BK)^\top P (A_{\sigma}+BK) \preceq \left(\frac{\gamma}{\psi_P(\epsilon)}\right)^2 P.
\end{align}
Putting \eqref{eqn:LMItilde} and \eqref{eqn:ApsiP} together, we arrive at \eqref{eqn:gammastarover}. 
\end{proof}

The result in Theorem \ref{thm:gamma} allows to establish a probabilistic stability certificate from the solution of the sampled problem (\ref{eqn:gammaomegaN}).

\begin{corollary}\label{cor:bound}
Given $N\in \mathbb{Z}^+$, let $\omega_N=\{(x_i,\sigma_i)\}_{i=1}^N$ be i.i.d with respect to the uniform distribution $\mathbb{P}$ over $\mathbb{S}\times \mathcal{M}$. Let $\gamma(\omega_N),P(\omega_N)$, and $K(\omega_N)$ be obtained from Algorithm \ref{algo:dataK}.
Then, for any $\epsilon\in (0,1)$, with probability no smaller than $1-\mathcal{B}(\epsilon;N)$, 
\begin{align}\label{eqn:gammastaroveromega}
	\rho(\mathcal{A}_{K(\omega_N)}) \le \frac{\gamma(\omega_N)}{\max\{\varphi_{P(\omega_N)}(\epsilon),\psi_{P(\omega_N)}(\epsilon)\}} 
\end{align}
where $\rho(\mathcal{A}_{K(\omega_N)})$ is defined in (\ref{eqn:rhoAK}) with $K = K(\omega_N)$ and $\mathcal{B}(\epsilon;N)$, $\varphi_{P(\omega_N)}(\epsilon)$ and $\psi_{P(\omega_N)}(\epsilon)$ are given in (\ref{eqn:mathcalB}), \eqref{eqn:varphiPepsilon} and \eqref{eqn:psiPepsilon} respectively.
\end{corollary}

\begin{proof}
From Lemma \ref{lem:omegaNcover}, with probability no smaller than $1-\mathcal{B}(\epsilon;N)$, $\omega_N$ is an $\epsilon$-covering of $\mathbb{S}\times \mathcal{M}$. Combining this with Theorem \ref{thm:gamma}, we obtain the result above, since Algorithm \ref{algo:dataK} always generates feasible iterations.
\end{proof}

\begin{remark}\label{rem:boundquad}
The results above bear some similarities with the probabilistic stability guarantees in \cite{ART:KBJT19,INP:RWJ21,INP:BJW21} which are concerned with autonomous systems, the major difference is that the bound in this paper is applicable to any feasible solution while \cite{ART:KBJT19,INP:RWJ21,INP:BJW21} rely on the optimality of the solution. Let us also highlight that the bound in Theorem \ref{thm:gamma} can be considered as an improvement to the one in \cite{INP:WBJ21} in the sense that the additional term $\psi_P(\epsilon)$ prevents the  bound  from going unbounded when $\kappa(P)(1-\cos(\delta^{-1}(\epsilon))) \ge 1$, which happens when $\epsilon$ is not sufficiently small. From a practical point of view, $\epsilon$ has to be small for the bound in \eqref{eqn:gammastarover} to be meaningful. In such cases,  $\varphi_P(\epsilon)$ is often larger than $\psi_P(\epsilon)$. Hence,  the bound that is really used in practice is $\frac{\gamma}{\varphi_P(\epsilon)}$ in \eqref{eqn:gammastarover}. 
\end{remark}

\begin{remark}
	In some practice situations, what we receive is a set of input-state data, i.e., $\{(x_i, u_i, x_i^+): i=1,2.\cdots, N\}$ where $x_i^+ = A_{\sigma_i} x_i + Bu_i$ and $u_i$ is the $i\textsuperscript{th}$ input. As $B$ is known, we can convert this data set into $\{(x_i,x_i^+ -Bu_i): i=1,2.\cdots, N\}$. We can then apply our approach on this converted data set. Furthermore, the states may not lie on the unit sphere. While the solution of the sampled problem in  (\ref{eqn:gammaomegaN}) does not change from a theoretical point of view, we can use the scaled data $\{(x_i/\|x_i\|,(x_i^+ -Bu_i)/\|x_i\|): i=1,2.\cdots, N\}$ to improve numerical stability. If the samples follow an isotropic Gaussian distribution centered at zero (with the covariance matrix being a scalar variance multiplied by the identity matrix) and are generated independently, the scaled points are uniformly distributed on the unit sphere and hence our probabilistic guarantees in Corollary \ref{cor:bound} are still valid.
\end{remark}

\subsection{Parallel computation}
To achieve high confidence in Theorem \ref{them:stabilitysos}, a large number of samples are typically needed. As a result, the problems (\ref{eqn:Kgamma}) and  (\ref{eqn:Pgamma}) at each iteration of Algorithm \ref{algo:dataK} quickly become demanding due to an increasing number of constraints. To circumvent this issue, inspired by the stochastic gradient descent and its variants \cite{BOO:BL17}, we propose a parallelized scheme for Algorithm \ref{algo:dataK} by dividing these constraints into small batches. More precisely, given $L\in \mathbb{Z}^+$, we build $L$ disjoint subsets of $\omega_N$, denoted by $\{\omega_N^i\}_{i=1}^L$, with $\cup_{i=1}^L \omega_N^i = \omega_N$. The choice of $L$ and $\{\omega_N^i\}_{i=1}^L$ depends on the number of computing resources available and their computation power. With this partition, we then solve the alternating minimization problems as defined in (\ref{eqn:Kgamma}) and (\ref{eqn:Pgamma}) individually for each batch, as shown in Algorithm \ref{algo:dataKparallel}. The solution of each subproblem provides a candidate descent direction. We then choose the solution that provides the lowest convergence rate via a line search heuristic (\ref{eqn:gammaell})-(\ref{eqn:hatgamma}), which guarantees feasibility of the iterate for all the constraints. The line search step also guarantees that $\{\gamma_k\}$ does not increase along iterations. Note that  (\ref{eqn:gammaell}) and (\ref{eqn:hatgamma}) are both scalar optimization problems that can be easily solved. The probabilistic guarantee in Corollary \ref{cor:bound} is still applicable as it only requires a feasible solution, though Algorithm \ref{algo:dataKparallel} may produce a more conservative solution $\gamma(\omega_N)$ compared with Algorithm \ref{algo:dataK}.

\begin{algorithm}[h]
	\caption{Parallel alternating minimization for quadratic stabilization}
	\begin{algorithmic}[1]
		\renewcommand{\algorithmicrequire}{\textbf{Input:}}
		\renewcommand{\algorithmicensure}{\textbf{Output:}}
		\REQUIRE $\{(x_i,A_{\sigma_i}x_i)\} _{i=1}^N$, $B$ , and some tolerance $\epsilon_{tol}>0$
		\ENSURE $\gamma(\omega_N),P(\omega_N)$, and $K(\omega_N)$\\
		\textit{Initialization}: Create a partition $\{\omega_N^\ell\}_{\ell=1}^L$ on $\omega_N$; Let $k\leftarrow 0$, $K_k \leftarrow 0$, $P_k \leftarrow I$, and $\gamma_k \leftarrow \max_{(x,\sigma)\in \omega_N} \frac{ \|A_\sigma x\|}{\|x\|}$; 
		\medskip
		\STATEx \textit{Minimization on $P$} \dotfill
		\medskip
		
		\FOR{$\ell=1,2,\cdots, L$}
		\STATE Solve $\bar{\mathcal{P}}(\omega_N^\ell;K_k)$ and let the solution be denoted by $P_k^\ell$;
		\STATE Compute
		\begin{align} \label{eqn:gammaell}
			\bar{\gamma}_k^\ell \leftarrow \min_{\lambda \in [0,1]} \max_{(x,\sigma)\in \omega_N} \frac{ \|A_\sigma x+BK_kx\|_{\bar{P}_k^\ell(\lambda)}}{\|x\|_{\bar{P}_k^\ell(\lambda)}}
		\end{align}
		where $\bar{P}_k^\ell(\lambda) \coloneqq (1-\lambda)P_k + \lambda P_k^\ell$ and let $\bar{\lambda}_k^\ell$ be the solution of (\ref{eqn:gammaell}).
		\ENDFOR
		\STATE Find the minimum among $\{ \bar{\gamma}_k^\ell\}$ and let $\bar{\ell}_k \coloneqq \arg \min_{\ell} \bar{\gamma}_k^\ell$;
		\STATE  Let $P_{k+1} \leftarrow \bar{P}_k^{\bar{\ell}_k}(\bar{\lambda}_k^{\bar{\ell}_k}) $;
		\medskip
		\STATEx \textit{Minimization on $K$} \dotfill
		\medskip
		\FOR{$\ell=1,2,\cdots, L$}
		\STATE Solve $\hat{\mathcal{P}}(\omega_N^\ell;P_{k+1})$ and let the solution be denoted by $K_k^{\ell}$;
		\STATE Compute 
		\begin{align} \label{eqn:hatgamma}
			\hat{\gamma}_{k}^\ell \leftarrow \min_{\lambda \in [0,1]} \max_{(x,\sigma)\in \omega_N} \frac{ \|A_\sigma x+B\hat{K}_{k}^\ell(\lambda)x\|_{P_{k+1}}}{\|x\|_{P_{k+1}}}
		\end{align}
		where $\hat{K}_{k}^\ell(\lambda)\coloneqq(1-\lambda) K_k+\lambda K_k^\ell$, and let $\hat{\lambda}_{k}^\ell$ denote the solution of (\ref{eqn:hatgamma});
		\ENDFOR
		\STATE Find the minimal among $\{ \hat{\gamma}_{k}^\ell\}$ and let $\hat{\ell}_{k} \coloneqq \arg \min_{\ell} \hat{\gamma}_{k}^\ell$;
		\STATE  Let $\gamma_{k+1} \leftarrow   \hat{\gamma}_{k}^{\hat{\ell}_{k}}$ and $K_{k+1} \leftarrow \mathcal{K}_{k}^{\hat{\ell}_{k}}(\hat{\lambda}_{k}^{\hat{\ell}_{k}})$;
		\medskip
		\STATEx \textit{Stopping criterion} \dotfill
		\medskip
		\IF {$\|\gamma_{k+1}-\gamma_k\|< \epsilon_{tol}$}
		\STATE $\gamma(\omega_N)\leftarrow \gamma_{k+1}$, $P(\omega_N)\leftarrow P_{k+1}$, $K(\omega_N)\leftarrow K_{k+1}$, and terminate;
		\ELSE
		\STATE Let $k\leftarrow k+1$ and go to Step 2.
		\ENDIF
	\end{algorithmic}
	\label{algo:dataKparallel}
\end{algorithm}

\section{SOS Lyapunov framework}\label{sec:sos}
Quadratic stabilization of switched systems can be very restrictive in some cases. To reduce conservatism, we can use sum of squares (SOS) techniques, which have already been used in \cite{ART:PJ08,ART:AJPR14} to improve the bound on the JSR. In the framework of data-driven stability analysis, the application of SOS optimization has already proven useful for the case of autonomous systems in \cite{INP:RWJ21}. Here, we want to show that SOS techniques are also applicable for the stabilization problem.

Let us first recall some definitions in SOS optimization. We refer to \cite{ART:PJ08} for the details. Given $x\in \mathbb{R}^n$ and $d\in \mathbb{Z}^+$, let $x^{[d]} \in \mathbb{R}^{{ n+d-1\choose d}})$ denote the $d$-lift of $x$ which consists of all possible monomials of degree $d$, indexed by all the possible exponents $\alpha$ of degree $d$
\begin{align}
	x^{[d]}_{\alpha} = \sqrt{\alpha !} x^{\alpha}
\end{align}
where $\alpha = (\alpha_1,\cdots,\alpha_n)$ with $\sum_{i=1}^n \alpha_i=d$ and $\alpha !$ denotes the multinomial coefficient 
\begin{align}
	\alpha ! \coloneqq \frac{d!}{\alpha_1 ! \cdots \alpha_n !}.
\end{align}
The $d$-lift of a matrix $A\in \mathbb{R}^{n\times n}$ is defined as: $A^{[d]}: x^{[d]} \rightarrow (Ax)^{[d]}$. The following proposition provides an upper bound for the JSR based on SOS Lyapunov functions.

\begin{proposition}[\cite{BOO:J09}, Thm. 2.13] \label{prop:jsrsos}
	Consider the closed-loop matrices $\mathcal{A}_K$ for some state feedback $K\in \mathbb{R}^{m\times n}$. For any $d\in \mathbb{Z}^+ (d\ge 1)$, if there exist $\gamma\ge 0$ and $P \succ 0$ such that, $ \forall\,A\in \mathcal{A_K}$, $x\in \mathbb{S}$, 
	\begin{align}
		((Ax)^{[d]})^\top P (Ax)^{[d]} \le \gamma^{2d} (x^{[d]})^\top P x^{[d]},
	\end{align}
	where $P\in \mathbb{R}^{D\times D}$ with $D = {n+d-1 \choose d}$, then $\rho(\mathcal{A}_K) \le \gamma$.
\end{proposition}

\subsection{Data-driven SOS stabilization}
In the model-free case, we formulate the following sampled problem using the given data set $\omega_N$:
\begin{subequations}\label{eqn:Psosomega}
	\begin{align}
	\mathcal{P}_d(\omega_N): 	&\min_{\gamma\ge 0, P\succeq I, K} \gamma\\
		\textrm{s.t.} \quad  & ((A_\sigma x+BKx)^{[d]})^\top P (A_\sigma x+BKx)^{[d]} \nonumber \\
		&\le \gamma^{2d} (x^{[d]})^\top P x^{[d]},\:  \forall\,(x,\sigma) \in \omega_N \label{eqn:ABKxPd}
	\end{align}
\end{subequations}
From the definition of $d$-lift of vectors above, this is a polynomial optimization problem with $mn+{ n+d-1\choose d} $ variables and $N$ polynomial constraints, which is much more computationally demanding than Problem (\ref{eqn:gammaomegaN}) depending on the degree $d$. For a small data set, this problem can be solved by polynomial toolboxes \cite{ART:L01,ART:HL03,sostools} or general nonlinear solvers, such as IPOPT \cite{INP:W09} and NLopt \cite{MISC:Johnson14}. For a large data set, to reduce the complexity, we again make use of the structure in (\ref{eqn:ABKxPd}) to develop an alternating minimization algorithm between $P$ and $K$ as Algorithm \ref{algo:dataK} for the quadratic case. For a given $K$, we define:
\begin{subequations}\label{eqn:Kd}
	\begin{align}
		\tilde{\mathcal{P}}_d(\omega_N;K): 	&\min_{\gamma, P\succeq I} \gamma\\
		\textrm{s.t.} \quad  & ((A_\sigma x+BKx)^{[d]})^\top P (A_\sigma x+BKx)^{[d]} \nonumber \\
		&\le \gamma^{2d} (x^{[d]})^\top P x^{[d]},\:  \forall\,(x,\sigma) \in \omega_N
	\end{align}
\end{subequations}
For a given $P$, we define:
\begin{subequations}\label{eqn:Pd}
	\begin{align}
		\hat{\mathcal{P}}_d(\omega_N;P): 	&\min_{\gamma\ge 0, K} \gamma\\
		\textrm{s.t.} \quad  & ((A_\sigma x+BKx)^{[d]})^\top P (A_\sigma x+BKx)^{[d]} \nonumber \\
		&\le \gamma^{2d} (x^{[d]})^\top P x^{[d]},\:  \forall\,(x,\sigma) \in \omega_N
	\end{align}
\end{subequations}
Indeed, one may observe that Problem (\ref{eqn:Kd}) can be solved efficiently by using SDP solvers (like Mosek \cite{INC:AA00}) and bisection on $\gamma$. While Problem (\ref{eqn:Pd}) is still a polynomial optimization problem, there are only $mn$ decision variables, which makes it easier to handle than Problem (\ref{eqn:Psosomega}). To solve the problems in (\ref{eqn:Kd}) and (\ref{eqn:Pd}), we typically need an initial solution, in particular for the polynomial problem (\ref{eqn:Pd}). In this alternating minimization algorithm, we set the initial solution to be the solution from the previous iteration, which leads to a convergent sequence of $\gamma$. The details of this procedure is given in Algorithm \ref{algo:dataKsos}.

\begin{algorithm}[h]
	\caption{Alternating minimization for SOS stabilization}
	\begin{algorithmic}[1]
		\renewcommand{\algorithmicrequire}{\textbf{Input:}}
		\renewcommand{\algorithmicensure}{\textbf{Output:}}
		\REQUIRE $\{(x_i,A_{\sigma_i}x_i)\}_{i=1}^N$, $B$, $d$ and some tolerance $\epsilon_{tol}>0$
		\ENSURE $\gamma_{sos}(\omega_N),P_{sos}(\omega_N)$, and $K_{sos}(\omega_N)$\\
		\textit{Initialization}: $k\leftarrow 0$, $K_k \leftarrow 0$, $P_k \leftarrow I$, and $\gamma_k \leftarrow \max_{(x,\sigma)\in \omega_N} \frac{ \|(A_\sigma x)^{[d]}\|}{\|x^{[d]}\|}$; 
		\STATE Obtain $P_{k+1}$ by solving $\tilde{\mathcal{P}}_d(\omega_N;K_k)$ via bisection on $\gamma$ starting from $ \max_{(x,\sigma)\in \omega_N} \frac{ \|(A_{\sigma} x+BK_kx)^{[d]}\|_{P_k}}{\|x^{[d]}\|_{P_k}}$;
		\STATE Obtain $\gamma_{k+1}$ and $K_{k+1}$ by solving $\tilde{\mathcal{P}}_d(\omega_N;P_{k+1})$ initialized at $K=K_k$;
		\IF {$\|\gamma_{k+1}-\gamma_k\|< \epsilon_{tol}$}
		\STATE $\gamma_{sos}(\omega_N)\leftarrow \gamma_{k+1}$, $P_{sos}(\omega_N)\leftarrow P_{k+1}$, $K_{sos}(\omega_N)\leftarrow K_{k+1}$;
		\STATE Terminate;
		\ELSE
		\STATE Let $k\leftarrow k+1$ and go to Step 2.
		\ENDIF
	\end{algorithmic}
	\label{algo:dataKsos}
\end{algorithm}

\subsection{Stability guarantees via sensitivity analysis}\label{sec:sosstable}
With the aforementioned alternating minimization algorithm for SOS stabilization, we get a feasible solution, denoted by $(\gamma_{sos}(\omega_N),P_{sos}(\omega_N), K_{sos}(\omega_N))$. Similar to Section \ref{sec:prob}, we now derive stability guarantees for this solution. However, due to the polynomial lifting in (\ref{eqn:Psosomega}), we lose convexity in the Lyapunov function, which prevents us from applying the geometric results in Section \ref{sec:prob} to the SOS framework. In particular, the reasoning about (\ref{eqn:LMItilde}) and (\ref{eqn:ApsiP}) in the proof of Theorem \ref{thm:gamma} does not hold. Hence, we use an alternative way to derive probabilistic guarantees on the solution obtained from Algorithm \ref{algo:dataKsos}.

Given any $\epsilon \in (0,1)$ and $d\in \mathbb{Z}^+$ ($d\ge 1$), let us define
\begin{align}\label{eqn:phiepd}
	\phi(\epsilon,d) \coloneqq \sum_{k=1}^{d} {d \choose k} \left( 2-2\cos\left(\delta^{-1}(\epsilon)\right) \right)^{\frac{k}{2}}
\end{align}
Similar to Theorem \ref{thm:gamma}, we derive a stability certificate for SOS stabilization as stated in the following theorem. 
\begin{theorem}\label{thm:AKsos}
Given a sample set $\omega_N \subset \mathbb{S}\times \mathcal{M}$ and an integer $d \ge 1$, consider Problem (\ref{eqn:Psosomega}). Let $(\gamma,P, K)$ be a feasible solution to Problem (\ref{eqn:Psosomega}). Suppose $\omega_N$ is an $\epsilon$-covering of $\mathbb{S}\times \mathcal{M}$ for some $\epsilon >0$. Then,
\begin{align} \label{eqn:rhoAKPsos}
	\rho(\mathcal{A}_{K}) \le & \sqrt[d]{\gamma^d  +  (\gamma^d+\bar{\rho}(\mathcal{A}_K)^d) \sqrt{\kappa(P)} \phi(\epsilon,d)}
\end{align}
where $\rho(\mathcal{A}_{K})$ is defined in (\ref{eqn:rhoAK}), $\phi(\epsilon,d)$ is given in (\ref{eqn:phiepd}), and
\begin{align}\label{eqn:barrhoA}
	\bar{\rho}(\mathcal{A}_K) \coloneqq \max_{A\in \mathcal{A}_K} \|A\|. 
\end{align}
\end{theorem}
\begin{proof}
We drop the subscript $N$ in $\omega_N$ in the proof for convenience. Since $\omega$ is an $\epsilon$-covering of $\mathbb{S}\times \mathcal{M}$, from Definition \ref{def:joint}, for any $(x,\sigma)\in \mathbb{S}\times \mathcal{M}$, there exists $(z,\sigma)\in \omega$ such that $|z^\top x| \ge \cos(\delta^{-1}(\epsilon))$, which implies that  $\|x-z\|\le \sqrt{2-2\cos(\delta^{-1}(\epsilon))}$ or $\|x+z\|\le \sqrt{2-2\cos(\delta^{-1}(\epsilon))}$. Due to the fact that $(-z,\sigma)$ also satisfies (\ref{eqn:ABKxPd}), we consider the case that $\|x-z\|\le \sqrt{2-2\cos(\delta^{-1}(\epsilon))}$ without loss of generality. From \cite{ART:PJ08}, it can be shown that 
\begin{align*}
		\|x^{[d]}-z^{[d]}\| &= \|x^{\otimes d}-z^{\otimes d}\| \le \sum_{k=1}^{d} {d \choose k} \|x-z\|^k \|z\|^{d-k} \\
		&\le \sum_{k=1}^{d} {d \choose k} \left( 2-2\cos(\delta^{-1}(\epsilon)) \right)^{\frac{k}{2}}
\end{align*}
where $x^{\otimes d}$ denotes the $d$-fold Kronecker product. With this and some manipulations, we get that
\begin{align*}
		&\|(A_\sigma x + BKx)^{[d]}\|_{P}\\
		\le & \|(A_\sigma z + BKz)^{[d]} + (A_\sigma + BK)^{[d]}(x^{[d]}-z^{[d]})\|_{P}\\
		\le & \|(A_\sigma z + BKz)^{[d]}\|_{P}\\
		&+\sqrt{\lambda_{\max}(P)} \|A_\sigma + BK\|^d \|x^{[d]}-z^{[d]}\|\\
		\le &  \gamma^d \|z^{[d]}\|_{P} + \sqrt{\lambda_{\max}(P)} \|A_\sigma + BK\|^d \phi(\epsilon,d)\\
		\le &  \gamma^d \|z^{[d]}\|_{P}+\sqrt{\lambda_{\max}(P)} \bar{\rho}(\mathcal{A}_K)^d \phi(\epsilon,d)\\
		\le &  \gamma^d \left(\|x^{[d]}\|_{P} + \sqrt{\lambda_{\max}(P)} \phi(\epsilon,d) \right)\\
		& +\sqrt{\lambda_{\max}(P)} \bar{\rho}(\mathcal{A}_K)^d \phi(\epsilon,d)\\
		\le &  \left(\gamma^d   + (\gamma^d+\bar{\rho}(\mathcal{A}_K)^d)  \sqrt{\kappa(P)} \phi(\epsilon,d) \right)\|x^{[d]}\|_{P}.
\end{align*}
From Proposition \ref{prop:jsrsos}, we then obtain (\ref{eqn:rhoAKPsos}).
\end{proof}
%From Algorithm \ref{algo:dataKsos},  it can be verified that $(\gamma_{sos}(\omega_N),P_{sos}(\omega_N), K_{sos}(\omega_N))$ is a feasible solution, i.e., 
%\begin{align}
%	&\|(A_\sigma x + BK_{sos}(\omega_N)x)^{[d]}\|_{P_{sos}(\omega_N)} \nonumber \\
%	\le & \gamma^d_{sos}(\omega_N) \|x^{[d]}\|_{P_{sos}(\omega_N)}, \forall (x,\sigma)\in \omega_N. \label{eqn:PsosKsos}
%\end{align}
%For convenience, we drop $\omega_N$ in  $\gamma_{sos}(\omega_N)$, $P_{sos}(\omega_N)$, and $K_{sos}(\omega_N)$ in the proof. 

\begin{remark}
	When $d=1$, it becomes exactly the quadratic case, which means that Theorem \ref{thm:AKsos} provides an alternative bound for the quadratic case. However, by numerical simulation, for reasonable values of $P$ and $\epsilon$, the bound in (\ref{eqn:gammastarover}) is better than the one in (\ref{eqn:rhoAKPsos}) with $d=1$. In addition, Theorem \ref{thm:AKsos} requires the information of $\bar{\rho}(\mathcal{A}_K)$, which is yet to be estimated.
\end{remark}

Following the same arguments as in Section \ref{sec:prob}, we can then establish probabilistic guarantees for this data-driven SOS framework. However, the bound in (\ref{eqn:rhoAKPsos}) relies on $\bar{\rho}(\mathcal{A}_K)$ which is not available. To handle this issue, we use the data set to estimate $\bar{\rho}(\mathcal{A}_K)$. Given a sample set $\omega_N \subset \mathbb{S}\times \mathcal{M}$ and any $K\in \mathbb{R}^{m\times n}$, we define the following problem:
\begin{subequations}\label{eqn:etaN}
\begin{align}
	\eta^*(\omega_N,K) \coloneqq &\min_{\eta \ge 0} \eta\\
	 \textrm{s.t.}\quad  &\|A_{\sigma}x+BKx\| \le \eta, \forall (x,\sigma) \in \omega_N.
\end{align}
\end{subequations}
An upper bound on $\bar{\rho}(\mathcal{A}_K)$ is then given in the following proposition. 

\begin{proposition}\label{prop:normest}
Given a sample set $\omega_N \subset \mathbb{S}\times \mathcal{M}$ and any $ K \in \mathbb{R}^{m\times n}$, let $\eta^*(\omega,K)$ be defined as in (\ref{eqn:etaN}). Suppose $\omega_N$ is an $\epsilon$-covering of $\mathbb{S}\times \mathcal{M}$ for some $\epsilon >0$. Then, 
	\begin{align}\label{eqn:barrhoeta}
	\bar{\rho}(\mathcal{A}_K) \le \frac{\eta^*(\omega,K)}{\cos(\delta^{-1}(\epsilon))}
\end{align}
where $\bar{\rho}(\mathcal{A}_K)$ is given in (\ref{eqn:barrhoA}).
\end{proposition}

\begin{proof}
	This result is a special case of Theorem \ref{thm:gamma} where $P=I$, which implies that $\varphi_P(\epsilon)  = \cos(\delta^{-1}(\epsilon))$.
\end{proof}

Based on the results above, we now present the main result of this section, which is a probabilistic stability certificate from the solution of the sampled problem (\ref{eqn:Psosomega}).

\begin{theorem}\label{them:stabilitysos}
	Given $N\in \mathbb{Z}^+$, let $\omega_N$ be i.i.d with respect to the uniform distribution $\mathbb{P}$ over $\mathbb{S}\times \mathcal{M}$. Let $(\gamma_{sos}(\omega_N),P_{sos}(\omega_N), K_{sos}(\omega_N))$ be obtained from Algorithm \ref{algo:dataKsos} and $\eta^*(\omega_N,K_{sos}(\omega_N))$ be defined as in (\ref{eqn:etaN}). For any $\epsilon \in (0,1)$, with probability no smaller than $1-\mathcal{B}(\epsilon;N)$, 
	\begin{align} \label{eqn:rhoAKPsoseta}
		&\rho(\mathcal{A}_{K_{sos}(\omega_N)}) \\
		\le & \gamma_{sos}(\omega_N) \sqrt[d]{ 1 + \left( 1 + \frac{\bar{\eta}(\omega_N)^d}{\gamma^d_{sos}(\omega_N)} \right)\sqrt{\kappa(P_{sos}(\omega_N))} \phi(\epsilon,d)} \nonumber
	\end{align}
where $\phi(\epsilon,d)$ is given in (\ref{eqn:phiepd}), $\mathcal{B}(\epsilon;N)$ is defined as in (\ref{eqn:mathcalB}), and 
\begin{align}
	\bar{\eta}(\omega_N) \coloneqq \frac{\eta^*(\omega_N,K_{sos}(\omega_N))}{\cos(\delta^{-1}(\epsilon))}
\end{align}
where  $\delta(\cdot)$ denotes the measure of  a \emph{symmetric spherical cap} as shown in (\ref{eqn:deltatheta}).
\end{theorem}
\begin{proof}
	From Lemma \ref{lem:omegaNcover}, $\omega_N$ is an $\epsilon$-covering of $\mathbb{S}\times \mathcal{M}$ with probability no smaller than $1-\mathcal{B}(\epsilon;N)$. Combining Theorem \ref{thm:AKsos}  and Proposition \ref{prop:normest}, we obtain the result above.
\end{proof}

\subsection{Parallelized scheme for SOS stabilization}
Similar to Algorithm \ref{algo:dataKparallel}, we can also use a parallelized scheme for SOS stabilization. In fact, it is even more beneficial to use a parallelized scheme as Algorithm \ref{algo:dataKsos} involves a large number of polynomial constraints of order $2d$ at each iteration. Suppose we build $L$ disjoint subsets $\{\omega_N^i\}_{i=1}^L$ from $\omega_N$ for a given $L\in \mathbb{Z}^+$. We solve the optimization problems in (\ref{eqn:Kd}) and (\ref{eqn:Pd})  with these subsets individually using SDP solvers (like Mosek \cite{INC:AA00}) and nonlinear solvers (like IPOPT \cite{INP:W09} and NLopt \cite{MISC:Johnson14}) respectively. At each iteration,  the solver is initialized with the solution from the last iteration, which, together with the line search heuristics as shown  in (\ref{eqn:tildegammasos}) and (\ref{eqn:hatgammasos}), guarantees a convergent sequence $\{\gamma_k\}_{k\in \mathbb{Z}^+}$. The parallelized scheme for SOS stabilization is given in Algorithm \ref{algo:dataKsosparallel}.

\begin{algorithm}[h]
	\caption{Parallel alternating minimization for SOS stabilization}
	\begin{algorithmic}[1]
		\renewcommand{\algorithmicrequire}{\textbf{Input:}}
		\renewcommand{\algorithmicensure}{\textbf{Output:}}
		\REQUIRE $\{(x_i,A_{\sigma_i}x_i)\} _{i=1}^N$, $B$  and some tolerance $\epsilon_{tol}>0$
		\ENSURE $\gamma_{sos}(\omega_N),P_{sos}(\omega_N)$, and $K_{sos}(\omega_N)$\\
		\textit{Initialization}: Create a partition $\{\omega_N^\ell\}_{\ell=1}^L$ on $\omega_N$; Let $k\leftarrow 0$, $K_k \leftarrow 0$ and $P_k \leftarrow I$; 
		\medskip
		\STATEx \textit{Minimization on $P$} \dotfill
		\medskip
		
		\FOR{$\ell=1,2,\cdots, L$}
		\STATE Solve $\tilde{\mathcal{P}}_d(\omega_N^\ell;K_k)$ and let the solution be denoted by $\tilde{P}(\omega_N^\ell;P_k)$;
		\STATE Compute
		\begin{align} \label{eqn:tildegammasos}
			\tilde{\gamma}_k^\ell \leftarrow \min_{\lambda \in [0,1]} \max_{(x,\sigma)\in \omega_N} \frac{ \|(A_\sigma x+BK_kx)^{[d]}\|_{\mathcal{P}_k^\ell(\lambda)}}{\|x^{[d]}\|_{\mathcal{P}_k^\ell(\lambda)}}
		\end{align}
	    where $\mathcal{P}_k^\ell(\lambda) \coloneqq (1-\lambda)P_k + \lambda \tilde{P}(\omega_N^\ell;P_k)$ and let $\tilde{\lambda}_k^\ell$ be the solution of (\ref{eqn:tildegammasos}).
		\ENDFOR
		\STATE Find the minimal among $\{ \tilde{\gamma}_k^\ell\}$ and let $\tilde{\ell}_k \coloneqq \arg \min_{\ell} \tilde{\gamma}_k^\ell$;
	    \STATE  Let $P_{k} \leftarrow \mathcal{P}_k^{\tilde{\ell}_k}(\tilde{\lambda}_k^{\tilde{\ell}_k})$;
	    \medskip
	    \STATEx \textit{Minimization on $K$} \dotfill
	    \medskip
		\FOR{$\ell=1,2,\cdots, L$}
		\STATE Solve $\hat{\mathcal{P}}_d(\omega_N^\ell;P_k)$ and let the solution be denoted by $\hat{K}(\omega_N^\ell;P_k)$;
		\STATE Compute 
		\begin{align} \label{eqn:hatgammasos}
			\hat{\gamma}_{k}^\ell \leftarrow \min_{\lambda \in [0,1]} \max_{(x,\sigma)\in \omega_N} \frac{ \|(A_\sigma x+B\mathcal{K}_{k}^\ell(\lambda)x)^{[d]}\|_{P_k}}{\|x^{[d]}\|_{P_k}}
		\end{align}
		where $\mathcal{K}_{k}^\ell(\lambda):=(1-\lambda) K_k+\lambda \hat{K}(\omega_N^\ell;P_k)$, and let $\hat{\lambda}_{k}^\ell$ denote the solution of (\ref{eqn:hatgamma});
		\ENDFOR
		\STATE Find the minimal among $\{ \hat{\gamma}_{k}^\ell\}$ and let $\hat{\ell}_{k} \coloneqq \arg \min_{\ell} \hat{\gamma}_{k}^\ell$;
		\STATE  Let $\gamma_{k+1} \leftarrow   \hat{\gamma}_{k}^{\hat{\ell}_{k}}$ and $K_{k+1} \leftarrow \mathcal{K}_{k}^{\hat{\ell}_{k}}(\hat{\lambda}_{k}^{\hat{\ell}_{k}})$;
	    \medskip
		\STATEx \textit{Stopping criterion} \dotfill
		\medskip
		\IF {$\|\gamma_{k+1}-\gamma_k\|< \epsilon_{tol}$}
		\STATE $\gamma_{sos}(\omega_N)\leftarrow \gamma_{k+1}$, $P_{sos}(\omega_N)\leftarrow P_{k+1}$, $K_{sos}(\omega_N)\leftarrow K_{k+1}$, and terminate;
		\ELSE
		\STATE Let $k\leftarrow k+1$ and go to Step 1.
		\ENDIF
	\end{algorithmic}
	\label{algo:dataKsosparallel}
\end{algorithm}

\section{Data-driven switched LQR}\label{sec:LQR}
In this section, we show that the proposed data-driven framework can be extended to infinite-horizon LQR problems of arbitrary switched linear systems. We consider the following infinite-horizon quadratic cost
\begin{align}
	J_{\infty}(\pmb{x},\pmb{u},\pmb{\sigma}) = \sum_{\ell=0}^\infty L(x(\ell),u(\ell))
\end{align}
where $\pmb{x}$, $\pmb{u}$ and $\pmb{\sigma}$ denote the state, control and switching sequences respectively, and $L(x,u) = x^\top Q x+u^\top R u$ is the stage cost with $Q \succ 0$ and $R \succ 0$. With these definitions, the infinite-horizon LQR problem of System (\ref{eqn:Asigma}) can be cast as
\begin{align}\label{eqn:Jstarx}
	J^*(x) \coloneqq \inf_{\pmb{u}} \sup_{\pmb{\sigma}\in \mathcal{M}^\infty} J_{\infty}(\pmb{x},\pmb{u},\pmb{\sigma}) 
\end{align}
with $x(0)=x$. Without any information on the switching signal, we only consider static linear feedback in the infinite-horizon LQR problem. Our goal is to find a quadratic upper bound $x^\top Px$ of $J^*(x)$ with a static feedback $u=Kx$, i.e., we want to find a pair $(K,P)$ such that $J_{\infty}(\pmb{x},\pmb{u},\pmb{\sigma}) \le \|x(0)\|_P^2$ for all $\pmb{\sigma}\in \mathcal{M}^\infty$ with $u(\ell) = Kx(\ell)$ for all $\ell \in \mathbb{Z}^+$. When $(K,P)$ satisfies 
\begin{align}\label{eqn:KPalpha}
	(A+BK)^\top P (A+BK) \preceq P - Q- K^\top RK
\end{align}
for all $A\in \mathcal{A}$, following standard manipulations (see, e.g., \cite{ART:LR06,ART:R06}), it can be shown that, for any switching sequence and any $k\in \mathbb{Z}^+$,
\begin{align}
	\|x(0)\|_P^2 - \sum_{\ell=0}^k L(x(\ell),u(\ell)) \ge \|x(k+1)\|_P^2
\end{align}
with $u(\ell) = Kx(\ell)$, which implies that $J^*(x) \le \|x\|_P^2$ for all $x\in \mathbb{R}^n$. For linear systems (when $ \mathcal{A}$ is a singleton), $(K,P)$ can be obtained by solving the Algebraic Riccati Equation, see \cite[Chapter 2]{BOO:LVS12} for details. For multiple dynamics matrices, using the Schur complement formula with $S = P^{-1}$ and $Y = KS$, a model-based solution of  (\ref{eqn:KPalpha})  can be obtained by solving the following LMI problem, see \cite[Theorem 1]{ART:KBM96},
\begin{subequations}\label{eqn:SQR}
	\begin{align}
		&\min_{S,Y} -\log\det(S)\\
	\textrm{s.t.}	\quad& \begin{pmatrix}
			S & SA^\top+Y^\top B^\top & S & Y^\top \\
			AS+BY & S & \pmb{0} & \pmb{0} \\
			S & \pmb{0} & Q^{-1} & \pmb{0}\\
			Y & \pmb{0} & \pmb{0} & R^{-1}
		\end{pmatrix} \succeq 0, \nonumber\\
	 & \qquad \forall A\in \mathcal{A}.
	\end{align}
\end{subequations}

\subsection{Sampled LQR problem}
In the model-free case, given a sample set $\omega_N \subset \mathbb{S}\times \mathcal{M}$, a sample-based relaxation of (\ref{eqn:KPalpha}) is given as follows
\begin{align}\label{eqn:ABPKQR}
	&(A_{\sigma}x+BKx)^\top P (A_{\sigma}x+BKx) \nonumber \\
	\le  &x^\top( P- Q-K^\top RK)x,  \forall (x,\sigma)\in \omega_N.
\end{align}
However, a solution to (\ref{eqn:ABPKQR}) may not be valid for  (\ref{eqn:KPalpha}) due to the randomness in the sampling. To deal with this issue, we introduce a scaling parameter $\xi\in (0,1)$ and formulate the following problem:
\begin{subequations}\label{eqn:PKQRomega}
	\begin{align}
		&\min_{P, K} ~ \textrm{tr}(P)\\
		\textrm{s.t.} \quad &(A_{\sigma}x+BKx)^\top P (A_{\sigma}x+BKx) \label{eqn:ABKPQRxi}\\
		& \le \xi^2 x^\top( P- Q-K^\top RK)x,  \forall (x,\sigma)\in \omega_N, \nonumber\\
		& P \succeq Q + K^\top R K \label{eqn:alphaPQRK}
	\end{align}
\end{subequations}
where the constraint (\ref{eqn:alphaPQRK}) is imposed as (\ref{eqn:ABKPQRxi}) does not guarantee that $P- Q-K^\top RK$ is positive semidefinite. When the data set is sufficiently rich, the problem (\ref{eqn:PKQRomega}) leads to the following robustness property.

\begin{theorem}\label{thm:Zxikappa}
Given a sample set $\omega_N \subset \mathbb{S}\times \mathcal{M}$, and $\xi \in (0,1)$, let $(P,K)$ be a feasible solution to Problem (\ref{eqn:PKQRomega}) and $Z = P -Q - K^\top R K$. Suppose $\omega_N$ is an $\epsilon$-covering of $\mathbb{S}\times \mathcal{M}$ and $Z$ is invertible. Then,
\begin{align}\label{eqn:ABKPZ}
	&(A_{\sigma}+BK)^\top P(A_{\sigma}+BK) \nonumber\\
	\preceq & \frac{\xi^2}{\max\{\varphi_Z(\epsilon),\psi_Z(\epsilon)\}^2} Z,  \quad \forall \sigma \in \mathcal{M},
\end{align}
where $\varphi_Z(\epsilon)$ and $\psi_Z(\epsilon)$ are given in \eqref{eqn:varphiPepsilon} and \eqref{eqn:psiPepsilon} respectively.
\end{theorem}

\begin{proof}
We drop the subscript $N$ in $\omega_N$ in the proof for convenience. The proof follows similar arguments as in Theorem \ref{thm:gamma}.  Consider the 
Cholesky decomposition of $Z= P -Q - K^\top R K = \bar{L}^\top \bar{L}$, define
\begin{align}
		\overline{\omega} \coloneqq \Big\{\Big(\frac{\bar{L}z}{\|\bar{L}z\|}, \sigma\Big): (z,\sigma) \in \omega \Big\} \subset \mathbb{S}\times \mathcal{M}.
\end{align}
From the same reasoning as in the proof of Theorem \ref{thm:gamma}, it can be shown that  $\overline{\omega}$ is an $\bar{\epsilon}$-covering of $\mathbb{S}\times \mathcal{M}$ with $\bar\epsilon=\delta(\bar\theta)$ and $\cos(\bar\theta) = \varphi_Z(\epsilon)= 1 - \kappa(Z) (1-\cos(\theta))$ where $\theta=\delta^{-1}(\epsilon)$. Again, as in (\ref{eqn:tildeomegasigma}), we define 
\begin{align}
		\overline{\omega}^{\sigma}\coloneqq\{x: (x,\sigma) \in \overline{\omega} \},\: \forall\, \sigma \in \mathcal{M}.
\end{align}
Following (\ref{eqn:deltaB}), we have $\varphi_Z(\epsilon) \mathbb{B} \subseteq \bigcap_{\sigma \in \mathcal{M}} \textrm{conv} (\pm\overline{\omega}^{\sigma})$. With Cholesky decomposition of $P=L^\top L$, it also holds that $\|\bar{A}_\sigma z\| \le \xi \|z\|$ for all $\sigma\in\mathcal{M}$ and $z \in \pm\overline{\omega}^{\sigma}$, where $\bar{A}_\sigma \coloneqq LA_{\sigma}\bar{L}^{-1}+LBK \bar{L}^{-1}$. Then, with (\ref{eqn:LMItilde}), we have, $\forall \sigma \in \mathcal{M}$,
\begin{align*}
		\frac{\varphi_Z(\epsilon)}{\xi} \bar{A}_\sigma \mathbb{B} \subseteq \frac{1}{\xi} \bar{A}_\sigma \textrm{conv} (\pm\overline{\omega}^{\sigma}) \subseteq \frac{1}{\xi} \textrm{conv} (\pm\bar{A}_\sigma  \overline{\omega}^{\sigma}) \subseteq \mathbb{B}.
\end{align*}
Hence, $\|\frac{\varphi_Z(\epsilon)}{\xi} \bar{A}_\sigma x \| \le \|x\|$ for any $x\in \mathbb{R}^n$ and $\sigma \in \mathcal{M}$,  which implies $\bar{A}_\sigma^\top \bar{A}_\sigma \preceq (\frac{\xi}{\varphi_Z(\epsilon)})^2 I$ for any $\sigma \in \mathcal{M}$. Finally, we get 
\begin{align}
	(A_{\sigma}+BK)^\top P(A_{\sigma}+BK) \preceq \frac{\xi^2}{\varphi_Z(\epsilon)^2} Z, \forall \sigma \in \mathcal{M}
\end{align}
which implies (\ref{eqn:ABKPZ}). Following the reasoning above and the arguments in the proof of Theorem \ref{thm:gamma}, we also have
\begin{align}
	(A_{\sigma}+BK)^\top P(A_{\sigma}+BK) \preceq \frac{\xi^2}{\psi_Z(\epsilon)^2} Z, \forall \sigma \in \mathcal{M}.
\end{align}
\end{proof}
From Theorem \ref{thm:Zxikappa}, we also impose the constraint $\max\{\varphi_{P -Q - K^\top R K}(\epsilon), \psi_{P -Q - K^\top R K}(\epsilon)\} \ge \xi$ in addition to (\ref{eqn:ABKPQRxi}) and (\ref{eqn:alphaPQRK}) to ensure that the pair $(K,P)$ is a feasible solution to (\ref{eqn:KPalpha}). As mentioned in Remark \ref{rem:boundquad}, in practice, we only need to consider $\varphi_{P -Q - K^\top R K}(\epsilon)$ as it is often larger for reasonable values of $\epsilon$. Motivated by this fact, we codesign $(P,K)$ and the parameter $\xi$ and modify Problem (\ref{eqn:PKQRomega}) by imposing an additional constraint as follows:
\begin{subequations}\label{eqn:PKQRomeganu}
	\begin{align}
		&\min_{P, K,\xi} ~ \textrm{tr}(P)\\
		\textrm{s.t.} \quad & (\ref{eqn:ABKPQRxi}), (\ref{eqn:alphaPQRK}),  \\
		& \kappa(P -Q - K^\top R K) \le \frac{1-\xi}{1-\cos(\delta^{-1}(\epsilon)}, \label{eqn:PQKRxi}\\
		& 0\le \xi \le \cos(\delta^{-1}(\epsilon)), \label{eqn:xidelta}
	\end{align}
\end{subequations}
where (\ref{eqn:PQKRxi}) is a reformulation of $\varphi_{P -Q - K^\top R K}(\epsilon) \ge \xi$,  (\ref{eqn:xidelta}) is due to the fact that $\kappa(P -Q - K^\top R K)\ge 1$, and $\epsilon$ is a user-defined parameter. Note that the constraint (\ref{eqn:PQKRxi}) is non-convex.

%Similar to (\ref{eqn:PKQRomegaLMI}), using Schur complement, we arrive at
%\begin{subequations}\label{eqn:PKQRomegaLMInu}
%	\begin{align}
%		&\min_{P, K, \nu > 0,\xi } \lambda_{\max}(P)\\
%		\textrm{s.t.} \quad & (\ref{eqn:ABKPQRxiLMI}),  (\ref{eqn:PQnu}), (\ref{eqn:PQnu}), (\ref{eqn:xidelta}),\\
%	&\begin{pmatrix}
%		 P-Q-\nu I & K^\top\\
%		K & R^{-1} 
%	\end{pmatrix} \succeq 0. \label{eqn:alphaPKRLMI}
%	\end{align}
%\end{subequations}

\subsection{Alternating LQR design}
To solve the non-convex problem (\ref{eqn:PKQRomeganu}), we also develop an alternating minimization algorithm.  While the general implementation is similar to the algorithms in the previous sections, the technical details are quite different due to the additional complexity arising from the constraint (\ref{eqn:PQKRxi}). As $K$ may not be a stabilizing feedback in the initial steps, the value of $\xi$ can be larger than $1$, which means that (\ref{eqn:PQKRxi}) is not valid. In view of this, we propose a heuristic in which we relax this constraint with an a-priori upper bound $\bar{\kappa}$ on $\kappa(P -Q - K^\top R K)$ and continuously minimize $\xi$ at each iteration. We then check the constraint (\ref{eqn:PQKRxi}) when $\xi$ is less than $1$. Given any $\epsilon \in (0,1)$ and $\alpha \ge 0$, let us define
\begin{align}\label{eqn:xistarep}
	\xi^*(\epsilon,\alpha)\coloneqq 1 - \alpha(1-\cos(\delta^{-1}(\epsilon))).
\end{align}
With this definition, it can be verified that, when $\kappa(P -Q - K^\top R K) \le \bar{\kappa}$ and $\xi \le \xi^*(\epsilon,\bar{\kappa})$, the constraint (\ref{eqn:PQKRxi}) is satisfied. 

We now present the overall procedure. At the initialization, we find the value of $\xi$ that is closest to $1$ such that (\ref{eqn:ABKPQRxi}) is feasible with $K=\pmb{0}$, denoted by $\xi_0$. Note that, given the a-priori upper bound $\bar{\kappa}$, the constraint $\kappa(P -Q - K^\top R K) \le \bar{\kappa}$ can be rewritten as $\lambda_{\max}(P -Q - K^\top R K) \le \bar{\kappa} \nu$ and $\lambda_{\min}(P -Q - K^\top R K) \ge \nu$ for some $\nu >0$, which are equivalent to $\nu I \preceq P -Q - K^\top R K \preceq \bar{\kappa} \nu I$. We then proceed by optimizing $(K,\xi)$ and $P$ in an alternating way. In the optimization of $(K,\xi)$, the goal is to decrease $\xi$ to a certain value below $1$ but not to minimize $\xi$ as much as possible. For this reason, we penalize the distance of the current $\xi$ to the previous value at each iteration in order to generate a smooth sequence $\{\xi_k\}$. Even when $P$ is fixed, the constraints  (\ref{eqn:ABKPQRxi}) and (\ref{eqn:alphaPQRK}) are still nonlinear. Again, we use the Schur complement formula to convert these constraints into LMIs as follows:
\begin{align}
& \begin{pmatrix}
		x^\top Px-x^\top Qx & (A_\sigma x+BKx)^\top P  & x^\top K^\top \\ 
		P(A_\sigma x+BKx) & \xi^2 P & \pmb{0}\\
		Kx & \pmb{0} & R^{-1}
	\end{pmatrix} \nonumber\\
	& \qquad \succeq 0,  \qquad    \forall (x,\sigma)\in \omega_N, \label{eqn:ABKPQRxiLMI}\\
	& \begin{pmatrix}
		P - Q & K^\top\\
		K & R^{-1}
	\end{pmatrix} \succeq   0.  \label{eqn:nualphaPLMI}
\end{align}
The optimization problem involving $(K,\xi)$ is formulated in (\ref{eqn:PKQRomegaLMInuK}). Replacing $\xi^2$ with a new variable $\zeta$ yields a convex problem. To ensure the constraint $\kappa(P -Q - K^\top R K) \le \bar{\kappa}$ at each iteration, the optimization of $(K,\xi)$ is followed by a backtracking step. After $K$ and $\xi$ are updated, we optimize over $P$ subject to the constraints (\ref{eqn:ABKPQRxi}),  (\ref{eqn:alphaPQRK}) and (\ref{eqn:PQKRKbark}). The details are given in Algorithm \ref{algo:dataKLQR}. This algorithm can also be parallelized following the same idea as in Algorithm \ref{algo:dataKparallel} and Algorithm \ref{algo:dataKsosparallel}.

\begin{algorithm}[h]
	\caption{Alternating minimization for LQR design}
	\begin{algorithmic}[1]
		\renewcommand{\algorithmicrequire}{\textbf{Input:}}
		\renewcommand{\algorithmicensure}{\textbf{Output:}}
		\REQUIRE $\{(x_i,A_{\sigma_i}x_i)\}_{i=1}^N$, $B$, $Q$,$R$, $\bar{\kappa}$, $c$, $\epsilon$ and $\epsilon_{tol}$
		\ENSURE $\xi(\omega_N),P(\omega_N)$, and $K(\omega_N)$\\
		\textit{Initialization}: Let $k\leftarrow 0$ and $K_k \leftarrow \pmb{0}$; Obtain $P_k$ and $\xi_k$ by solving
		\begin{subequations}\label{eqn:PKQRomegaLMinit}
			\begin{align}
				&\min_{P,\xi} \xi\\
				\textrm{s.t.} \quad & (\ref{eqn:ABKPQRxi}), (\ref{eqn:alphaPQRK}),\\
				& \kappa(P -Q-K^\top R K) \le  \bar{\kappa} \label{eqn:PQKRKbark}\\
				& \xi \ge 1 \label{eqn:xik}
			\end{align}
		\end{subequations}
	    with $K=K_k$.
		\STATE Solve
		\begin{subequations}\label{eqn:PKQRomegaLMInuK}
			\begin{align}
				(\bar{K}_k,\bar{\xi}_k) \leftarrow &\arg\min_{K,\xi} \xi^2+ c(\xi^2-\xi^2_k)^2\\
				\textrm{s.t.} \quad & (\ref{eqn:ABKPQRxiLMI}), (\ref{eqn:nualphaPLMI}),\\
				& \xi \ge \xi^*(\epsilon,\bar{\kappa}),
			\end{align}
		\end{subequations}
	    with $P = P_{k}$;
	    \STATE Obtain the stepsize via backtracking
	    \begin{subequations}
	    \begin{align}
	    	\lambda_k \leftarrow &\arg\min_{\lambda\in [0,1]} \lambda: \\
	    	\textrm{s.t.} \quad &\kappa(P_k-Q-\mathcal{K}_k(\lambda)^\top R\mathcal{K}_k(\lambda)) \le \bar{\kappa}
	    \end{align}
       \end{subequations}
   where $\mathcal{K}_k(\lambda) \coloneqq \lambda K_k+(1-\lambda)\bar{K}_k)$.
        \STATE Let $K_{k+1} = \lambda_k K_k+(1-\lambda_k)\bar{K}_k$ and $\xi_{k+1} = \sqrt{\lambda_k \xi^2_k+(1-\lambda_k)\bar{\xi}_k^2}$;
		\STATE Obtain $P_{k+1}$ by solving
		\begin{subequations}\label{eqn:PKQRomegaLMInuP}
			\begin{align}
				&\min_{P} ~ \textrm{trace}(P)\\
				\textrm{s.t.} \quad &  (\ref{eqn:ABKPQRxi}),  (\ref{eqn:alphaPQRK}), (\ref{eqn:PQKRKbark})
			\end{align}
		\end{subequations}
		with $K=K_{k+1}$ and $\xi = \xi_{k+1}$.
		\IF {$\|\xi_{k+1}-\xi_k\|< \epsilon_{tol}$ or $\kappa(P_{k+1} -Q-K_{k+1}^\top R K_{k+1}) \le \frac{1-\xi_{k+1}}{1-\cos(\delta^{-1}(\epsilon)}$}
		\STATE $\xi(\omega_N)\leftarrow \xi_{k+1}$, $P(\omega_N)\leftarrow P_{k+1}$, $K(\omega_N)\leftarrow K_{k+1}$;
		\STATE Terminate;
		\ELSE
		\STATE Let $k\leftarrow k+1$ and go to Step 1.
		\ENDIF
	\end{algorithmic}
	\label{algo:dataKLQR}
\end{algorithm}

Based on Theorem \ref{thm:Zxikappa}, a probabilistic guarantee is presented below for the proposed data-driven LQR.
\begin{corollary}
	Consider Problem (\ref{eqn:KPalpha}) with $Q,R \succ 0$. Given $N\in \mathbb{Z}^+$, let $\omega_N=\{(x_i,\sigma_i)\}_{i=1}^N$ be i.i.d with respect to the uniform distribution $\mathbb{P}$ over $\mathbb{S}\times \mathcal{M}$. For any $\epsilon\in (0,1)$, we define $\xi^*(\epsilon)$ as in (\ref{eqn:xistarep}). Let $(\xi(\omega_N),P(\omega_N),K(\omega_N))$ be a solution obtained from Algorithm \ref{algo:dataKLQR} and $Z(\omega_N)\coloneqq P(\omega_N) -Q - K(\omega_N)^\top R K(\omega_N)$. Then, with probability no smaller than $1-\mathcal{B}(\epsilon;N)$, the following statement holds: if $\xi(\omega_N) \le \xi^*(\epsilon,\kappa\left(Z(\omega_N)\right))$, $(P(\omega_N),K(\omega_N))$ is a feasible solution to (\ref{eqn:KPalpha}).
\end{corollary}
\begin{proof}
From Lemma \ref{lem:omegaNcover}, with probability no smaller than $1-\mathcal{B}(\epsilon;N)$, $\omega_N$ is an $\epsilon$-covering of $\mathbb{S}\times \mathcal{M}$. Then, according to Theorem \ref{thm:Zxikappa}, when $\xi(\omega_N) \le \xi^*(\epsilon,\kappa\left(Z(\omega_N)\right))$, $(P(\omega_N),K(\omega_N))$ is a feasible solution.
\end{proof}

\begin{remark}
	From (\ref{eqn:PKQRomegaLMInuK}),  the sequence $\{\xi_k\}$ obtained from Algorithm \ref{algo:dataKLQR} is monotonically non-increasing. With the constraint in (\ref{eqn:xik}), the sequence $\{\xi_k\}$ is bounded from below and thus is convergent. 
\end{remark}

\section{Numerical experiments}\label{sec:num}
In this section, we demonstrate the proposed data-driven control framework on several numerical examples. The codes are written in Julia and available at  \href{https://github.com/zhemingwang/DataDrivenSwitchControl}{https://github.com/zhemingwang/DataDrivenSwitchControl}.
 
\subsection{Quadratic stabilization for  multiple examples}
We consider quadratic stabilization for several systems. We first consider an example that can be quadratically stabilized. Consider the following switched linear system with $n=2,m=1$ and $M=3$:
\begin{align*}
	A_1 &= \begin{pmatrix}
		0.7 & 0.16\\
		 1.1 & -1.1
	\end{pmatrix},
	A_2 = \begin{pmatrix}
0.4 & -0.84\\
 0.83 & 0.35
	\end{pmatrix},\\
	A_3 &= \begin{pmatrix}
0.37 & 0.96\\
 0.34 & -1.2
	\end{pmatrix},
	B = \begin{pmatrix}
-0.9\\
 -1.2
	\end{pmatrix}.
\end{align*}
To show that this is not a trivial example, we compute $\rho(\mathcal{A}) =1.544$, the JSR of the open-loop system, using the JSR toolbox \cite{INP:VHJ14}. We also compute $\gamma^* = 0.8756$ as defined in (\ref{eqn:PK}) by solving (\ref{eqn:QY}) with bisection on $\gamma$ to check feasibility of quadratic stabilization for this example.

First, let $N = 2000$ and set the confidence level to $\mathcal{B}(\epsilon;N) = 0.01$. The corresponding value of $\epsilon$ can be computed via bisection using (\ref{eqn:mathcalB}). With this setting, the upper bound in (\ref{eqn:gammastarover}) is valid with probability larger than $99\%$.
For convenience, let
\[
\overline{\gamma}(\omega_N)\coloneqq\frac{\gamma(\omega_N)}{1-\kappa(P(\omega_N))(1-\cos(\delta^{-1}(\epsilon)))}.
\]
We then sample $\omega_N$ according to the uniform distribution on $\mathbb{S} \times \mathcal{M}$ and apply Algorithm \ref{algo:dataK} with the tolerance being $\epsilon_{tol} = 10^{-3}$. The obtained solution is:
\begin{align*}
\gamma(\omega_N) &=0.8836,\: K(\omega_N) =\begin{pmatrix}
0.6626 &  -0.4359
\end{pmatrix},\\
P(\omega_N) &= \begin{pmatrix}
1.1302 & 0.5480\\
0.5480 & 3.3064
\end{pmatrix}.
\end{align*}
The bound obtained from (\ref{eqn:gammastarover}) is $\overline{\gamma}(\omega_N)= 0.8873$. To empirically verify the solution, we compute $ \rho(\mathcal{A}_{K(\omega_N)}) = 0.8767$, the JSR of the closed-loop system with $K(\omega_N)$, using the JSR toolbox \cite{INP:VHJ14}.

We then apply the proposed approach to higher dimensional examples. Again, the confidence level is set to $0.01$, i.e., $\mathcal{B}(\epsilon;N) = 0.01$. We choose different values of $N$ and compute the corresponding $\epsilon$ that satisfies $\mathcal{B}(\epsilon;N) = 0.01$. The dynamics matrices $\mathcal{A}$ and the input matrix $B$ are generated randomly in a way that each entry is chosen from the uniform distribution over $[-1,1]$. Note that these random examples may not be stabilizable by a static linear feedback. Hence, in the simulation, we only compute the upper bound $\overline{\gamma}(\omega_N)$ and compare it to the true solution $\gamma^*$ defined in Problem (\ref{eqn:PKS}). Finally, we divide the data set into a number of subsets of size $1000$ and use the parallelized scheme in Algorithm \ref{algo:dataKparallel}. The results are shown in Figure \ref{fig:bound}. As expected, the sample size needed to reach the true white-box solution increases as the system dimension and the number of modes increase.

\begin{figure}[h]
\centering
\includegraphics[width=0.8\linewidth]{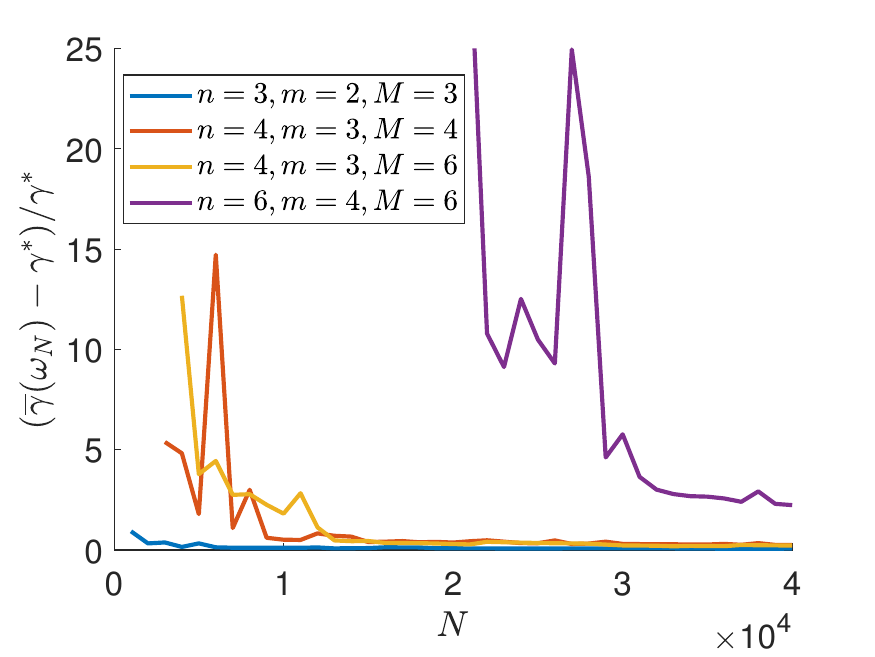}
\caption{Convergence of the sample-based solution to the true solution for systems of different dimensions and modes.}
\label{fig:bound}
\end{figure}

\subsection{Quadratic stabilization versus SOS stabilization}
For stability analysis of switched linear systems, SOS Lyapunov functions often provide tighter bounds than quadratic Lyapunov functions, see \cite{ART:PJ08,ART:AJPR14} for a few examples. Recently, we have also found out that it is beneficial to use SOS Lyapunov functions in data-driven stability analysis in  \cite{INP:RWJ21} with numerical examples. Here, we show that this is also the case in the stabilization problem. We give such an example where SOS stabilization  outperforms quadratic stabilization. Consider the following switched linear system with $2$ modes, which is generated randomly with a procedure similar as above:
\begin{align*}
	A_1 &= \begin{pmatrix}
		-1.6856 & -0.1665\\
		 0.7785 &-1.6321
	\end{pmatrix}, 	B = \begin{pmatrix}
	0.1975\\
	 0.8640
\end{pmatrix},\\
	A_2 &= \begin{pmatrix}
-0.2915 & -3.2824\\
 3.9761 &-0.02274
	\end{pmatrix}.
\end{align*}
For different values of $N$,  let $\epsilon$ be chosen such that $\mathcal{B}(\epsilon;N) = 0.01$. We then compute the upper bounds on the JSR according to the results in Section \ref{sec:prob} and Section \ref{sec:sosstable}. For SOS stabilization, we consider the case of $d=2$. To speed up the computation, we use parallelized schemes in Algorithm \ref{algo:dataKparallel} and Algorithm \ref{algo:dataKsosparallel} where the size of the subsets is set to be $1000$.
The results are given in Figure \ref{fig:quadvssos}, which shows that the JSR upper bound from SOS stabilization becomes tighter as $N$ increases. When $N=25000$, the solutions for the two algorithms are: $K(\omega_N) = [-2.0257 ~ 0.5407]$ and $K_{sos}(\omega_N) = [-2.2517 ~ 0.9063]$. In addition to the probabilistic bounds in Figure \ref{fig:quadvssos}, we also compute the actual JSR of the closed-loop system using the JSR toolbox \cite{INP:VHJ14}: $\rho(\mathcal{A}_{K(\omega_N)}) = 2.6002$ and $\rho(\mathcal{A}_{K_{sos}(\omega_N)}) = 2.1188$. As a reference, we also solve the white-box quadratic stabilization problem as given in (\ref{eqn:QY}) and obtain the white-box solution $K^*=[-0.2831 ~ -0.2965]$. With this, we again compute the actual JSR \cite{INP:VHJ14}: $\rho(\mathcal{A}_{K^*})=2.5582$, as indicated by the dashed line in Figure \ref{fig:quadvssos}. From these computations, we can see that the data-driven solution from SOS stabilization even outperforms the solution from the white-box quadratic stabilization. 

\begin{figure}[H]
	\centering
	\includegraphics[width=0.8\linewidth]{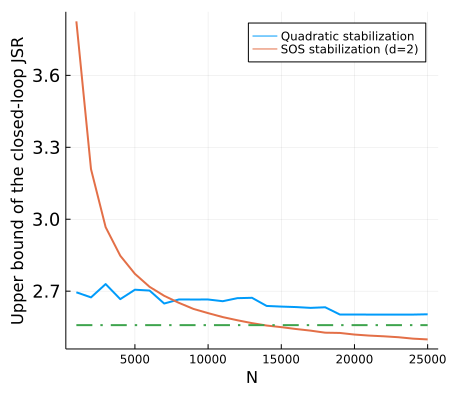}
	\caption{Comparison of quadratic stabilization and SOS stabilization: the dashed line is the JSR of the white-box quadratic stabilization solution.}
	\label{fig:quadvssos}
\end{figure}

\subsection{Building temperature regulation with LQR}
Buildings can be viewed as complex control systems with both continuous and discrete dynamics. For instance, events like opening/closing doors and windows instantly affect the dynamic evolution of the zone temperature. To capture these hybrid behaviors, hybrid RC networks are often used in the modeling of building systems
\cite{ART:FLC16,INP:ALCL17}. Consider a building with three zones as shown in Figure \ref{fig:building}. Its thermal RC model is described as below, $i=1,2,3$,
\begin{align*}
	c_i \dot{T}_i = \sum\limits_{j\not= i}\frac{T_{j}-T_{i}}{R_{ji}}+ \frac{T_o-T_i}{R^o_{i}} + m_ic_p(T_s-T_i) + q_i
\end{align*}
where $i$ is the index of the zone, $T_i$ is the temperature of zone $i$, $T_o$ is the temperature of outside air, $c_i$ is the thermal capacitance of the air in zone $i$, $R_{ij}$ denotes the thermal resistances between zone $i$ and zone $j$, $R^o_i$ denotes the thermal resistance between zone $i$ and the outside environment, $c_p$ is the specific heat capacity of air, $T_s$ is the temperature of the supply air delivered to zone $i$, $m_i$ is the flow rate into zone $i$ and $q_i$ is the thermal disturbance from internal loads like occupants and lighting. The temperatures of the supply air and the outside environment are known, i.e., $T_s$ and $T_o$ are available. The thermal disturbance is estimated as: $q_1 = 0.1 \si{kJ/s},q_2 = 0.1 \si{kJ/s}, q_3 = 0.12 \si{kJ/s}$. Other system parameters are given in Table \ref{tab:para}. The control objective is to steer the temperature of each zone to $T_{target} = \SI{24}{\degreeCelsius}$.

\begin{figure}[h]
	\centering
	\includegraphics[width=0.9\linewidth]{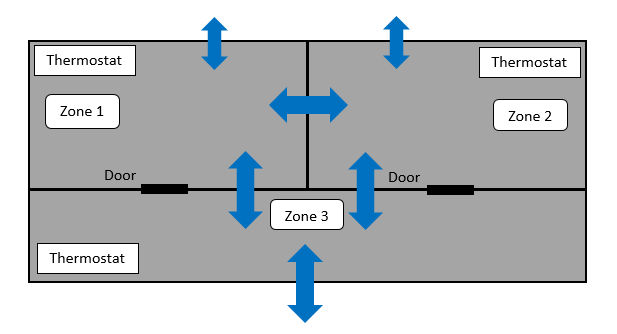}
	\caption{System schematic of the building}
	\label{fig:building}
\end{figure}

\begin{table}[H]
	\centering
	\renewcommand{\arraystretch}{1.5}
	\begin{tabular}{|c|c|c|}
		\hline
		Symbol & Value & Units \\ \hline
		$c_i,\forall i$ & $1.375\times 10^3$ & $\si{kJ/K}$ \\ \hline
		$c_p$ & $1.012$ & $\si{kJ/(kg\cdot K)}$ \\ \hline
		$R_{12}=R_{21}$ & $1.5$ & $\si{K/kW}$ \\ \hline
		$R^o_1 = R^o_2$ & $3$ & $\si{K/kW}$ \\ \hline
		$R^o_3$ & $2.7$ & $\si{K/kW}$ \\ \hline
		$T_s$ & $16$ & $\si{\degreeCelsius}$ \\ \hline
		$T_o$ & $32$ & $\si{\degreeCelsius}$ \\ \hline
	\end{tabular}
	\caption{System parameters} \label{tab:para}
\end{table}

By letting $Q^{AC}_i = m_ic_p(T_s-T_i)$, an equivalent linearized model is obtained:
\begin{align*}
	c_i \dot{T}_i = & \sum\limits_{j\not= i}\frac{T_{j}-T_{i}}{R_{ji}}+ \frac{T_o-T_i}{R^o_{i}} + Q^{AC}_i + q_i, i = 1,2,3.
\end{align*}
The steady input for the given targeted temperature is $\bar{Q}^{AC}_i = -\frac{T_o-T_{target}}{R^o_{i}} - q_i$ for all $i$. Let $x_i = T_i-T_{target}$ and $u_i = Q^{AC}_i-\bar{Q}^{AC}_i$ for all $i$. We get
\begin{align*}
	\dot{x}_i =  \sum\limits_{j\not= i}\frac{x_{j}-x_{i}}{c_i R_{ji}}- \frac{x_i}{c_i R^o_{i}} + \frac{1}{c_i}u_i, i = 1,2,3.
\end{align*}
As the doors are frequently and unpredictably open and closed, this is a typical switching system in which the thermal resistances $R_{13}$ (or $R_{31}$) and $R_{23}$ (or $R_{32}$) are changing arbitrarily. For each door, we consider two modes: "open" and "closed". Hence, for the overall system, there are $4$ modes in total, see Table \ref{tab:R}. The values of these thermal resistances are assumed to be unknown. In the simulation, we only use them to generate synthetic data.
\begin{table}[H]
	\centering
	\renewcommand{\arraystretch}{1.5}
	\begin{tabular}{|c|c|c|}
		\hline
		& "open" & "closed" \\ \hline
		$R_{13} = R_{31} ~ (\si{K/kW})$ & $0.8 $ & $1.2$ \\ \hline
		$R_{23} = R_{32} ~ (\si{K/kW})$ & $0.8 $ & $1.2$ \\ \hline
	\end{tabular}
	\caption{Thermal resistances of different modes.} \label{tab:R}
\end{table}
We discretize the continuous-time system with the sampling time $\tau = 3$ minutes. Here, we use the Euler forward method for its simple implementation. See \cite{ART:KMB13} for an extensive study on different discretization methods for buildings. The discretized system is given below:
\begin{align*}
	&\frac{x_i(t+1) -x_i(t)}{\tau} \\
	= &\sum\limits_{j\not= i}\frac{ x_{j}(t)-x_{i}(t)}{c_i R_{ji}}- \frac{x_i(t)}{c_i R^o_{i}}+\frac{u_i(t)}{c_i}, i = 1,2,3.
\end{align*}
We now design the LQR for the switching system above in a data-driven fashion. Let the LQR parameters be $Q=I$ and $R=0.02I$. For different values of $N$, we generate the data set $\omega_N$. The confidence level is set to be  $\mathcal{B}(\epsilon;N) = 0.01$ in all the cases and the corresponding $\epsilon$ is computed via bisection. We then use Algorithm \ref{algo:dataKLQR} with $\bar{\kappa}=100$ to obtain a feasible solution to the sampled LQR problem in (\ref{eqn:PKQRomega}). Let
\begin{align*}
	\bar{\xi}(\omega_N) \coloneqq \frac{\xi(\omega_N)}{\xi^*\left(\epsilon,\kappa\left(P(\omega_N) -Q - K(\omega_N)^\top R K(\omega_N)\right)\right)}
\end{align*}
From the discussions in Section \ref{sec:LQR}, the value $\bar{\xi}(\omega_N)$ can be considered as an indicator for feasibility (provided that $\omega_N$ is an $\epsilon$-covering of of $\mathbb{S}\times \mathcal{M}$):  $(P(\omega_N),K(\omega_N))$ is a feasible solution when $\bar{\xi}(\omega_N) \le 1$. We show the values of $\bar{\xi}(\omega_N)$ as the size of the data set $\omega_N$ increases in Figure \ref{fig:xibar}. From this curve, we can see that $\bar{\xi}(\omega_N)$  becomes less than $1$ when $N \ge 8000$. We also show the LQR solution when $N=12000$ below
\begin{align*}
	K(\omega_N)& = \begin{pmatrix}
		-3.3773 & -0.5579 & -0.6681\\
		-0.5580 &-3.3763 & -0.6660\\
		-0.6683 & -0.6686 & -3.2397
	\end{pmatrix}\\
	P(\omega_N) &= \begin{pmatrix}
		1.4041 & 0.1138 & 0.1333\\
		0.1138 & 1.4041 & 0.1333\\
		0.1333 & 0.1333 & 1.3787
	\end{pmatrix}.
\end{align*}
\begin{figure}[h]
	\centering
	\includegraphics[width=0.9\linewidth]{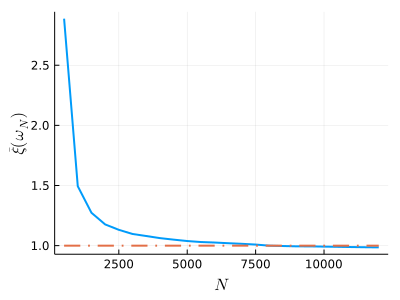}
	\caption{Feasibility measure of the LQR solution}
	\label{fig:xibar}
\end{figure}

As a way to validate the sample-based solution, we compute the white-box LQR solution by solving the LMI inequalities in (\ref{eqn:SQR}) and the solution is given below
\begin{align*}
	K^*& = \begin{pmatrix}
		-3.1736 & -0.4840 & -0.5938\\
		-0.4840 &-3.1736 & -0.5938\\
		-0.5882 & -0.5882 & -3.0320
	\end{pmatrix},\\
	P^* &= \begin{pmatrix}
		1.3844 & 0.1085 & 0.1270\\
		0.1085 & 1.3844 & 0.1270\\
		0.1270 & 0.1270 & 1.3602
	\end{pmatrix}.
\end{align*}
The relative differences between the two solutions are given by $\|K(\omega_N)-K^*\|/\|K^*\| \times 100\% = 8.44 \%$ and
$\|P(\omega_N)-P^*\|/\|P^*\| \times 100\% = 1.93 \%$, which suggests that the sample-based solution is a quite good approximation. Finally, we show the evolution of the temperatures of the three zones with the controller $u = K(\omega_N)x$ in Figure \ref{fig:T}.
\begin{figure}[h]
	\centering
	\includegraphics[width=0.9\linewidth]{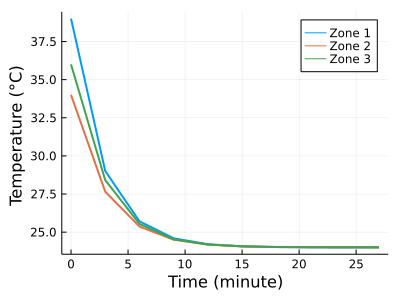}
	\caption{Temperatures of the three zones with the LQR controller.}
	\label{fig:T}
\end{figure}

\section{Conclusions}
We have presented a data-driven control framework for stabilization of black-box switched linear systems in which the dynamics matrices and the switching signal are unknown using quadratic and SOS Lyapunov functions. With quadratic Lyapunov functions, the stabilization problem is formulated as a biconvex problem using a finite number of trajectories. With SOS Lyapunov functions, we end up with a nonlinear optimization problem with a set of polynomial constraints. We then propose alternating minimization algorithms to solve these problems by making use of the underlying structure. In both cases, we develop parallelized schemes that allow to handle high-dimensional systems. Using the notions of covering/packing numbers, we also provide probabilistic stability guarantees via geometric analysis for the quadratic Lyapunov technique and sensitivity analysis for the SOS Lyapunov technique. Finally, we show that the proposed data-driven framework can be extended to LQR design of switched linear systems.

%\appendix
%\section*{Derivation of Problem (\ref{eqn:PKIomega})}
%As $P$ is invertible, (\ref{eqn:gammaomegaNxp}) can be rewritten as
%\begin{align}
%     (A_\sigma x+BKx)^\top P P^{-1} P (A_\sigma x+BKx) \le \gamma^2 x^\top Px, \nonumber\\
%     \forall\,(x,\sigma) \in \omega_N.
%\end{align}
%Then, from the Schur complement \cite{BOO:BV04}, it becomes (\ref{eqn:PKIomegaxp}). The reason that the inequality $P\succ 0$ can be replaced by $P \succeq I$ is due to the fact that the feasibility of $P$ implies the feasibility of $P/\lambda_{\min}(P)$. 

%\bibliographystyle{model2-names}
\bibliographystyle{unsrt}
\bibliography{Reference}

\end{document}